\documentclass[12pt,reqno]{amsart}
\usepackage{amssymb}
\usepackage{txfonts}
\usepackage{amsfonts}
\usepackage{amsmath}
\usepackage{graphicx}
\usepackage{hyperref}
\usepackage{float}
\usepackage{epstopdf}
\usepackage{color}

\allowdisplaybreaks
\newtheorem{theorem}{Theorem}[section]
 \newtheorem{definition}[theorem]{Definition}

 \newtheorem{lemma}[theorem]{Lemma}
 \newtheorem{proposition}[theorem]{Proposition}
 \newtheorem{corollary}[theorem]{Corollary}

\numberwithin{equation}{section}
\def\FRAME#1#2#3#4#5#6#7#8
{
 \begin{figure}[ptbh]
 \begin{center}
 \includegraphics[height=#3]{#7}
 \caption{#5}
 \label{#6}
 \end{center}
 \end{figure}
}

\begin{document}

\title{Dirichlet forms  and critical exponents on fractals}

\pagestyle{plain}
\author {Qingsong Gu}
\address{Department of Mathematics\\ The Chinese University of Hong Kong, Hong Kong, China}
\email{qsgu@math.cuhk.edu.hk}
\author {Ka-Sing Lau}
\address{Department of Mathematics\\ The Chinese University of Hong Kong, Hong Kong, China\\
\& School of Mathematics and Statistics, Central China Normal University, Wuhan, China\\
\& Department of Mathematics, University of Pittsburgh, Pittsburgh, Pa. 15217, USA}
\email{kslau@math.cuhk.edu.hk}

\thanks {The research is supported in part by the HKRGC grant.}

\maketitle
\begin{abstract}  Let $B^{\sigma}_{2, \infty}$ denote the Besov space defined on  a compact set $K \subset {\Bbb R}^d$ which is equipped with an $\alpha$-regular measure $\mu$. The {\it critical exponent} $\sigma^*$ is the supremum of the $\sigma$ such that $B^{\sigma}_{2, \infty} \cap C(K)$ is dense in $C(K)$.  It is well-known that for many standard self-similar sets $K$, $B^{\sigma^*}_{2, \infty}$  are the domain of some local regular Dirichlet forms. In this paper, we explore new situations that the underlying fractal sets admit inhomogeneous resistance scalings, which yield two types of critical exponents. We  will restrict our consideration on the p.c.f. sets. We first develop a technique of quotient networks to study the general theory of these critical exponents. We then construct two asymmetric p.c.f. sets, and use them to illustrate  the theory and examine the function properties of the associated Besov spaces at the critical exponents; the various Dirichlet forms on these fractals will also be studied.
\end{abstract}

\bigskip

\section{\bf Introduction}

\medskip

Let $K$ be a closed subset in ${\Bbb R}^d$ with the Euclidean metric, and let $\mu$ be an $\alpha-$regular measure on $K$, that is, there exists $\alpha>0$ such that for any ball $B(x,r)$ with $0<r<\text{diam}(K)$,
\begin{equation} \label {eq1.1}
\mu(B(x,r))\asymp r^\alpha.
\end{equation}
(Here $f\asymp g$ means there exists constant $C>0$ such that
$C^{-1}g\leq f\leq Cg.$)\  Fix $\sigma>0$, for $u\in L^2(K,\mu)$, let
\begin{equation} \label {eq1.2}
[u]^2_{B_{2,\infty}^\sigma}:=\sup\limits_{0<r<1}r^{-\alpha-2\sigma}
\int_K\int_{B(x,r)}|u(x)-u(y)|^2d\mu(y)d\mu(x),
\end{equation}
and define  $B^\sigma_{2,\infty}:=\{u\in L^2(K, \mu):||u||_{B^\sigma_{2,\infty}}<\infty\}$ with norm
$||u||_{B^\sigma_{2,\infty}}:=||u||_2+ [u]_{B_{2,\infty}^\sigma} .$
The space is a Banach space and belongs to the class of Besov spaces (cf., for example \cite{GHL},  \cite{J}, \cite {JW}; note that this space is also denoted by Lip$(\sigma, 2, \infty)$.)

\medskip

Obviously, $B^\sigma_{2, \infty} \subset B^{\sigma'} _{2, \infty}$ if $0<\sigma' < \sigma$. The space $B^\sigma_{2, \infty}$ can be dense in $C(K)$, or dense in $L^2(K, \mu)$; it can also become trivial as $\sigma$ increases, depending on the geometry of $K$ and $\mu$.  Let us define the {\it critical exponents} on  $(K,\mu)$ by
\begin{equation*}
\sigma^*:=\sup\big\{\sigma:\ B^\sigma_{2,\infty}\cap C(K) \text{ is dense in $C(K)$}\big\}\
 .
\end{equation*}
For many self-similar sets $K$, $B^{\sigma^*}_{2,\infty}$ are the domains of some local regular Dirichlet forms (if exist), and they are essential in the study of the Laplacians, Brownian motions, and the associated heat kernels  \cite{B, GHL, HW, J, K, Pi, S}.  The value $\beta^* = 2 \sigma^*$ is called the {\it walk dimension} of $(K,\mu)$. It is an important parameter in the study of  heat kernels, which corresponding to the speed of diffusion on the underlying sets. Heuristically, the larger the value $\beta^*$, the harder is for the diffusion process (Brownian motion) to drift away from the initial position. It is well-known that if $K$ is a domain in ${\Bbb R}^d$, then $\sigma^*=1$; if $K$ is the $d$-dimensional Sierpinski gasket, then $\sigma^*= \log (d+3)/(2\log 2)$ \cite{J}. There are  extensions on the nested fractals \cite{Pi}, and approximate value of the Sierpinski carpet \cite{BB}; also for Cantor-type sets, $\sigma^*= \infty$ \cite{KLW}.  More generally, the notion of Besov space and critical exponents have been extended to  metric measure spaces $(K, d, \mu)$, where $(K, d)$ is a locally compact, separable metric space, and $\mu$ is $\alpha$-regular as before. It is known that with a heat kernel assumption and a  {\it chain condition} on $(X, d)$,  we have $1 \leq \sigma^* \leq\frac 12 (\alpha +1)$ \cite{GHL}.

\bigskip

 In the previous studies of  Dirichlet forms on $B^{\sigma^*}_{2, \infty}$, one often assumes that the space admits a Brownian motion with a Gaussian or a sub-Gaussian heat kernel. In such cases, it is known that if $\sigma > \sigma^*$,  then $B^\sigma_{2, \infty}$ consists of constant functions only.  For this we define another critical exponent
\begin{equation*}
 \sigma^\#:=\sup\big\{\sigma: \ B^\sigma_{2,\infty} \  \hbox {contains non-constant functions} \big\},
\end{equation*}
Clearly, $\sigma^* \leq \sigma^\#$,  and  for the standard examples, we always have  $\sigma^* =\sigma^\# $. In this paper, we will bring up the different situations through two asymmetric self-similar sets, and give a detail study of the two critical exponents as well as  the functional behaviors of the associated Besov spaces. The investigation is intended to get a better understanding of the local regular  Dirichlet forms, of which the existence is still not clear on the more general fractal sets.

\bigskip

Our consideration is on the {\it post critically finite} (p.c.f.) self-similar sets \cite{K}. Let $\{F_i\}_{i=1}^N$ be an iterated function system (IFS) of the form $F_i(x) = \rho(x-b_i)+b_i$ with $0<\rho<1,\ b_i \in {\Bbb R}^d$, and let $K$ be the self-similar set. Assume that the IFS has the p.c.f. property \cite {K}, let $V_0$ be the boundary of $K$,  $V_n = \bigcup_{i=1}^NF_i(V_{n-1})= \bigcup_{|\omega| =n} F_{\omega} (V_0) $, and $V_*= \bigcup_{n=1}^\infty V_n$. We  write $V_\omega = F_{\omega} (V_0)$, and for a function $u$ on $V_n$, we define
\begin {equation} \label {eq1.3}
E_n[u] = {\sum}_{x,y\in V_\omega ;\ |\omega| = n}|u(x)-u(y)|^2,
\end{equation}
  and call it a {\it primal energy} on $V_n$. According  to  Jonsson \cite {J} and Bodin \cite {Bo} (see also \cite{GL}), we have the following  discrete expression of the Besov semi-norm.

\medskip

\begin{proposition}\label{th1.1} For  a p.c.f. self-similar set defined by the IFS as above, and for  $2\sigma>\alpha$ we have
\begin{equation} \label {eq1.4}
[u]^2_{B^\sigma_{2,\infty}}\asymp
\sup_{j\geq 0}\Big \{ \rho^{-(2\sigma-\alpha)j}E_j [u]\Big\}.
\end{equation}
\end{proposition}

\bigskip
In the notion of electrical network, the primal energy form $E_n$ corresponds to  a network $G_n$ on $V_n$ with unit resistance on each pair of vertices in $V_\omega$ ($|\omega|=n$) in  $V_n$.  We denote by $R_n(p, q),\  p, q \in V_0$,   the trace (or induced resistance) of $G_n$ on $V_0$. For example, on the Sierpinski gasket, we have $R_n(p, q) = r^{-n} = \Big (\frac 53\Big)^n$ for all $p, q \in V_0$, and for the Vicsek cross, $R_n(p, q) = r^{-n} = 3^n$; in these cases, for any $n\geq 1$ and $u$ on $V_0$, there is a minimal energy extension of $u$ on $V_*$ (and hence on $K$) such that $r^{-n}E_n [u]= E_0[u]$, and for $u$ on $V_*$, the sequence  $\{r^{-n}E_n [u]\}_n$ is increasing, so that ${\mathcal E} [u]:= \lim_{n\to \infty}r^{-n}E_n [u]$ exists, and $\mathcal E$ is the classical local regular Dirichlet form on $K$. This $r$ is called a {\it renormalizing factor}, and  $r = \rho^{(2\sigma^*-\alpha)}$ in \eqref{eq1.4}.

\bigskip

In this paper, we will study the case that the traces $R_n(p, q),  p, q \in V_0$ have different growth rates, which give rise to the two critical exponents of the Besov spaces on $K$.  Our first task is to  develop the needed theoretical background  for the two critical exponents. Roughly speaking, $\sigma^*$ is determined by the minimum growth rate of the sequences $\{R_n(p,q)\}_n$ for $p, q\in V_0$ (Theorem \ref{th4.1}), and $\sigma^\#$ is determined by the maximum of such growth rates (Theorem \ref{th4.3}).  We also  provide some criteria of the density of  $B^{\sigma^*}_{2, \infty}$ and $B^{\sigma^\#}_{2, \infty}$ in $C(K)$ and $L^2(K, \mu)$ respectively (Theorem \ref{th4.1}(i)-(ii), Propositions \ref{th4.2} and \ref{th4.4}). One of the major techniques in this study is to define an {\it equivalent relation} to partition  the $V_n$'s according to the growth rates of $R_n(p, q),\  p, q \in V_0$,  and consider the equivalent classes and the quotient network.  From the electrical network point of view,  taking quotient means  shorting the circuit (putting zero resistance) at the vertices in the equivalent classes.  The quotient network can simplify calculations,  and give interesting properties and different geometries, in particular, the growth rate of the traces of the original network can be modified to the need on the quotient.  We give a detail study of the structure of the quotient network in Section 3 (Theorem \ref {th3.2}): we show that the equivalent classes are related to the attractors of a  graph directed system \cite {F, MW},  and give a sufficient condition for $R_n(p, q)$ to be a renormalization factor  localized on certain equivalent classes (Lemma \ref{th3.4}).

\medskip

We remark that the device of quotient network was first considered by Sabot \cite {Sa} with a different definition ($G$-relation, $G$ for group), which is used to study the existence and uniqueness of Dirichlet forms on symmetric ramified self-similar fractals. We will draw some comparison of the two in Section 3.
\bigskip

We will  present two asymmetric p.c.f. sets that the traces $R_n(p, q), \ p, q \in V_0$, have different growth rates. We  make special emphasis on the constructive aspects to illustrate the new situations. The  first example is modified from the Vicsek cross by adding two eyebolts on the cross to produce the irregularity  (see Figure \ref{fig5}), we call it the {\it  eyebolted Vicsek cross}. It consists of $21$ maps  with contraction ratio $1/9$, and has four boundary points $V_0$. By equipping the $V_n$  with the primal energy, and using  a generalized $\Delta$-Y transform from the electrical network theory, we show that the traces $R_n(p,q)$  have different growth rates, but the same power of growth $9^n$ (Proposition \ref{th5.2}).  By using this we conclude that (Theorems \ref{th5.4} and \ref{th6.1}).

\medskip

 \begin{theorem}\label{th1.2}  For the eyebolted Vicsek cross $K$ in Figure \ref{fig5}, the critical exponents are
\begin{equation*}
\sigma^* =\sigma^\# = \frac 12 \Big (1+\frac{\log21}{\log9}\Big ).
\end{equation*}
Moreover,

\vspace {0.1cm}

\ (i) $B^{\sigma^*}_{2, \infty} \ (\subset  C(K))$ \ is dense in $L^2(K, \mu)$, but not dense in $C(K)$;

\vspace{0.1cm}

(ii) there are two kinds of  (non-primal) local regular Dirichlet forms that can be constructed on  $K$, one satisfies the energy self-similar identity; the other follows from a ``reverse recursive construction".  Their domains are different from
$B^{\sigma^*}_{2, \infty}$.
\end{theorem}

\medskip

From (i) and Proposition \ref{th1.1}, we see that we can not have a {\it regular} (sufficiently many continuous functions) Dirichlet form from the renormalized limit of the primal energy and  has $B^{\sigma^*}_{2, \infty}$ as its domain. On the other hand, in (ii), we can use different conductances to obtain  energy forms that yield local regular Dirichlet forms on $K$. One construction gives an energy form that satisfies the {\it energy self-similar identity} \cite {K};   it  provides a concrete constructive proof to implement the abstract proof (fixed point theory) for the existence of such energy form on  asymmetric p.c.f. sets \cite {L,Me,Sa,HMT, Pe}. The other construction,  we call it {\it reverse recursive method}, is to fix an initial data at $V_0$, and iterate this to $V_n$ to obtain a sequence of compatible networks. This method  first appeared in a probabilistic study by Hattori, Hattori and Watanabe \cite {HHW} on the Sierpinski gasket $K$ ($abc$-gasket),  they showed that there is an asymptotically one-dimensional diffusion process on $K$. Some further development and extensions can be found in \cite {HJ, HK, HK1, HY} by Hambly {\it et al}, and in \cite {GLQ} by the authors.

\medskip

 We call the second example a {\it Sierpinski sickle}. It is a connected p.c.f. set $K$ generated by an IFS of $17$ similitudes of contraction ration $1/7$ (see Figure \ref{fig7});   the boundary $V_0$ has three points. By using the  $\Delta$-Y transform, we show that the  traces  $R_n(p,q), \ p, q \in V_0$ are comparable to $7^n$ and $(17/2)^n$ (Proposition \ref{th5.7}). By using this, we conclude that (Theorems \ref{th5.8}, and \ref{th6.2}),

\medskip

\begin{theorem} \label {th1.3} For the Sierpinski sickle (Figure \ref{fig7}), we have
\begin{align*}
\sigma^* = \frac 12 \Big (1+\frac{\log17}{\log7}\Big ), \qquad
\sigma^\# =\frac 12 \Big (\frac{2\log17-\log2}{\log7}\Big )\ .
\end{align*}
Moreover,

\vspace {0.1cm}
\ (i) $B^{\sigma^*}_{2, \infty} \ (\subset C(K)) $  is dense in $C(K)$, and $B^{\sigma^\#}_{2, \infty}$ is dense in $L^2(K, \mu)$.

\vspace {0.1cm}
(ii) there are (non-primal) Dirichlet forms on $K$ that satisfy the energy self-similar identity; but the reverse recursive method does not yield a Dirichlet form on $K$.
\end{theorem}

\bigskip
We remark that not all asymmetric p.c.f. set $K$ will give inhomogeneous rate on the $R_n(p,q)$'s. In fact, in the above two examples, the construction is quite delicate; if we make  small variances on the IFS, then the growth rate of the $R_n(p,q)$'s on $K$ will have the same power (as in the Theorem \ref{th1.2}), and there are Dirichlet forms with energy self-similar identities, and have  $B^{\sigma^*}_{2, \infty}$ as domain.
We will discuss this in the remark section.

\bigskip

For the organization of the paper, in Section 2, we recall some basic definitions, and make some comments of Proposition \ref{th1.1} on the discretization of the Besov norm.  We also introduce the notion of the trace (induced resistance) and the  $\Delta$-Y transform. In Section 3, we introduce the compatible equivalent relations on the network induced by the primal energy, and study the structure of the quotient network. In Section 4, we make use of this to prove some theoretical results for the two critical exponents.  We present the two examples in Section 5,  and  prove the first part of Theorems \ref{th1.2} and \ref{th1.3}. The last part of the two theorems on the construction of Dirichlet forms are given in Section 6. In Section 7, we give some remarks of the two examples, and discuss briefly on another related Besov space ${B}^{\sigma}_{2,2}$ of the non-local Dirichlet forms. We also give an Appendix of the directed graph self-similar sets that is associated with the quotient networks in Section 3.

\bigskip

\section{\large \bf Preliminaries}

\bigskip

We first recall the  definition of a Dirichlet form. Let $(M, d)$ be a locally compact, separable metric space, and let $\nu$ be a Radon measure on $M$ with supp$(\nu) = M$; the triple $(M, d, \nu)$ is called a {\it metric measure space}. Let $C_0(M)$ denote the space of continuous functions  with compact support.
\medskip

\begin{definition}\label{de2.1}
 On  $(M, d, \nu)$,
 a {\rm  Dirichlet form}  $\mathcal E$ with domain ${\mathcal F}$ is a symmetric bilinear form  which is non-negative definite, closed, densely defined  on $L^2(M, \nu)$, and satisfies the Markovian property: $u\in {\mathcal F} \Rightarrow \tilde u := (u\vee 0)\wedge 1 \in \mathcal F$ and ${\mathcal E} [\tilde u] \leq {\mathcal E} [u]$. (Here  $\mathcal E [u] := {\mathcal E}(u, u)$ denote the {\rm energy} of $u$.)

\vspace{0.1cm}
A Dirichlet form is called {\rm regular} if ${\mathcal F}\cap C_0(M)$ is dense in $C_0(M)$ with the supremum norm, and dense in ${\mathcal F}$ with the ${\mathcal E}^{1/2}_1$-norm.  It is called {\rm local} if ${\mathcal E}(u, \upsilon) =0$ for $u, \upsilon \in {\mathcal F}$  having  disjoint compact  supports.
\end{definition}
\medskip

\medskip
The importance of a local regular Dirichlet form is that it  induces a {\it Laplacian} on $M$. However it is a non-trivial matter to construct or to prove the existence of such form. In fact there are only limited class of self-similar sets on which the existence of Laplacians is known. Throughout we will consider the specific class of p.c.f. self-similar sets \cite {K}, which is defined in the following.  Unless otherwise specify, we will assume $M =K $ as a compact subset in ${\Bbb R}^d$ and $d$ is the Euclidean metric.

\bigskip

Let $\{F_i\}_{i=1}^N$  be an IFS on ${\Bbb R}^d$ such that
\begin{equation}\label{eq2.1}
F_i(x) = \rho(x -b_i) + b_i, \quad 1\leq i \leq N,
\end{equation}
where $0<\rho <1$ and $b_i \in {\Bbb R}^d$.   Let  $K= \bigcup_{i=1}^N F_i (K)$ be the self-similar set, and let $\mu$ be the self-similar measure defined by $\mu = \frac 1N\sum_{i=1}^N \mu \circ F_i^{-1}$.  If the IFS  satisfies the open set condition (OSC), i.e., there is a nonempty bounded open set $O$ such that $F_i (O) \subset O$ and $F_i(O) \cap F_j(O) = \emptyset$ for $i\not = j$, then the Hausdorff dimension of $K$ is ${\rm dim}_H (K) = \alpha = \frac {\log N}{|\log\rho|}$, and $\mu$ is the $\alpha$-Hausdorff measure normalized on $K$, it is $\alpha$-regular
in the sense of \eqref{eq1.1}. Without loss of generality, we always assume that $K$ is connected.

\medskip

We define the symbolic space of $K$ as usual. Let $\Sigma= \{1, \cdots N\}$ be the alphabets, $\Sigma^n$ the set of words of length $n$,  and $\Sigma^\infty$  the set of infinite words $\omega = \omega_1 \omega_2 \cdots$;  let  $\pi : \Sigma^\infty \to K$  be defined by
$ \{x\} =\{\pi(\omega)\}=\bigcap_{n\geq1}K_{\omega_1\cdots \omega_n}$, a symbolic representation of $x\in K $ by $\omega$.

\medskip

Following Kigami \cite{K},  we define the
\emph{critical set} $ \mathcal{C} $ and the \emph{post-critical set } $\mathcal{P}$ for $K$  by
$$
\mathcal{C}=\pi^{-1}\Big({\bigcup}_{1 \leq i < j \leq N }\big( K_i \cap K_j\big )\Big),\quad \mathcal{P}={\bigcup}_{m\geq 1}\tau^m(\mathcal{C}),
$$
where $K_i = F_i(K)$, $\tau :\Sigma^\infty \to \Sigma^\infty$ is the left shift by one index. If  $\mathcal{P}$ is a finite set, we call $\{F_i\}_{i=1}^N$ a {\it post-critically finite} (p.c.f.) IFS, and $K$ is a p.c.f. self-similar set.  The  {\it boundary}  of $K$ is defined to be  $V_0= \pi(\mathcal{P})$. (We always assume  $\# (V_0) \geq 2$ to avoid triviality.) We also define
$$
V_n=\bigcup_{i \in \{1,\ldots,N\}}F_i(V_{n-1}),\quad V_*=\bigcup_{n\geq 1}V_n.
$$
It is clear that $\{V_n\}_{n=0}^\infty$ is an increasing sequence of sets, and $K$ is the closure of $V_*$. We call $V_\omega :=F_\omega(V_0)$  a \emph{cell} of $V_n$ for any $\omega \in \Sigma^n$, where $F_\omega =F_{\omega_1}\circ \cdots \circ F_{\omega_n}$.

\medskip

It is known that a p.c.f. IFS in \eqref{eq2.1}  satisfies the open set condition \cite{DL} (More generally, this is true if the associate similar matrices $A_i$ of $F_i$ (instead of the $\rho$ in \eqref{eq2.1}) are commensurable i.e., there exists $A$ such that $A_i = A^{n_i}$; but it is not true without this assumption \cite {TKV}.) Hence the p.c.f. self-similar set $K$ has dimension $\alpha$, and is associated with a self-similar measure $\mu$ that is $\alpha$-regular.

\bigskip

For a Besov space $B^\sigma_{2, \infty}$ on  a compact set $K$ with an $\alpha$-regular measure, we recall a continuity property of its functions (\cite{GHL}, over there the following proposition is put under the assumption that a heat kernel exists, but it was not used in the proof).

\medskip

\begin {proposition} \label{th2.2}  For $2\sigma > \alpha$, then the identity map $\iota :
B_{2,\infty}^\sigma \to  C^{(2\sigma-\alpha)/2} (K)$ is a continuous embedding. (Here $C^\beta(K)$ denotes the class of Lipschitz functions on $K$.)
\end {proposition}

\medskip
The discretized version of a Besov space in Proposition \ref{th1.1} was first established by  Jonsson \cite {J} on the Sierpinski gasket, and  he showed that  the critical exponent  $\sigma^* =\frac{ \log 5}{2 \log 2}$.  He also introduced the notion of {\it regular triangular system} (RTS) on the $d$-sets ($d$ is $\alpha$ here) to study the piecewise linear bases of the Besov space generated by this system \cite {J1}. In \cite {Bo}, Bodin extended Jonsson's discretization theorem to the Besov spaces $B^\sigma_{p,q}, 1\leq p,q \leq \infty$, for a $d$-set that admits a RTS.  He  stated without proof that similar to the RTS case,  the discretization is also true for p.c.f. sets. Actually there are  technical steps that need to be justified, and they are provided in  \cite {GL}.  For our purpose, we will need Proposition \ref{th1.1} in a slightly more general form.

\bigskip

\begin{corollary}\label{th2.3} With the p.c.f. self-similar set defined  as above,  then  for  $2\sigma>\alpha$ and for any integer $\ell >0$,
\begin{equation} \label {eq2.2}
[u]^2_{B^\sigma_{2,\infty}}\asymp
\sup_{j\geq 0}\Big \{ \rho^{-(2\sigma-\alpha)
\ell j}\sum_{x,y\in F_\omega (V_0);\ |\omega| =\ell j}|u(x)-u(y)|^2\Big\}.
\end{equation}
\end{corollary}

\medskip
We will make frequently use of the following proposition to construct functions in $B^{\sigma}_{2, \infty}$\  \cite {GL}.

\medskip
\begin{proposition}\label{th2.4} Assume $2\sigma>\alpha$,
then for any function $u$ on $V_*$ , if $u$ satisfies
$$
{\sup}_{j\geq0}\Big \{ \rho^{-(2\sigma-\alpha)j}{\sum}_{x,y\in V_\omega,\ |\omega|=j} |u(x) -u(y) |^2\Big \} < \infty,
$$
 $u$ can be extended continuously to $\tilde u$ on $K$, and $\tilde u \in B^{\sigma}_{2, \infty}$.
\end{proposition}

\medskip

  Let $\ell(V_*)$ denote the class of real-valued functions on $V_*$. For $u \in \ell(V_*)$,  we  define an energy form ${\mathcal E}_n[u]$ on  $V_n \ , \ n\geq 0$,  by
\begin{equation} \label {eq2.3}
{\mathcal E}_n[u] = \sum_{x,y\in F_\omega (V_0),\ |\omega| = n} c_n(x,y)|u(x)-u(y)|^2,
\end{equation}
where $c_n(x,y)$ is the conductance of the nodes $x, y$.  In literature, the most studied approach to construct a Dirichlet form on a p.c.f.  set is to consider the sequence  ${\mathcal E}_{n+1}[u] = \sum_{i=1}^N {\tau_i}^{-1}\ {\mathcal E}_n[u\circ F_i]$,  where  $0<\tau_i<1$ are the renormalization factors. If  ${\mathcal E} [u]= \lim_{n\to \infty} {\mathcal E}_n[u]$ exists for all $u \in \ell (V_*)$, then  ${\mathcal E}$ satisfies the {\it energy self-similar identity}
\begin{equation} \label {eq2.4}
\mathcal E[u] = \sum_{i=1}^N {\tau_i}^{-1}\ {\mathcal E}[u\circ F_i] \ ,  \qquad  u \in {\mathcal F},
\end{equation}
 and defines a local regular Dirichlet form  on $L^2(K,\mu)$ for a given Radon measure $\mu$ fully supported on $K$ \cite{K, S}.  If all the  $\tau_i$ are equal, then   the  Dirichlet form $\mathcal E$  on the metric measure space  $(K, |\cdot |, \mu)$ has domain  ${\mathcal F} = B_{2, \infty}^{\sigma^*}$ ($\mu$ is the normalized $\alpha$-Hausdorff measure on $K$). If the $\tau_i$'s are not all equal,  then we can consider the metric measure space $(K, d_r, \nu)$, where  $d_r$ is the resistance metric on $K$,  and $\nu$ is the self-similar measure with  weights  $\{\tau_i^s\}_{i=1}^N$ where $\sum_{i=1}^N {\tau_i}^s =1$, and the domain ${\mathcal F}$ is a modified Besov space with respect to $(K, d_r, \nu)$ (\cite {K}, \cite {Pi}, \cite {GHL}, \cite {HW}).

\bigskip

Let  $G_n :=(V_n, r_n)$ denote the corresponding electrical network of \eqref{eq2.3} with resistance  $r_n (x,y) = c_n(x,y)^{-1}$, $x, y \in V_n$ as resistance. It is known that \cite [Theorem 2.1.6] {K} for any  $ m<n$, there is an induced network of $G_n$ on $V_m$ with resistance $R_{n, m} (x, y)$ such that for $u \in \ell(V_m)$,
\begin {equation}\label {eq2.4'}
\min \big \{{\mathcal E}_n[\upsilon]:  \upsilon \in \ell (V_n), \ \upsilon|_{V_m} = u\big \} = \sum_{x,y \in V_m} \frac 1{R_{n,m}(x,y)}|u(x) -u(y)|^2.
\end{equation}

\medskip

Let $\{R_n\}_{n=0}^\infty$ be an increasing  sequence of  positive real numbers, suppose there exists $R> 1$ such that  for any $\varepsilon >0$, there exists $N(\varepsilon)$ such that for all  $n\geq N(\varepsilon)$,
\begin{equation*}
 R^{(1-\varepsilon)n}\leq R_n\leq  R^{(1+\varepsilon)n},
\end{equation*}
 then we call $R$ the {\it asymptotic geometric growth rate}  of $R_n$.

 \medskip

\begin{definition} \label {de2.5}  We call $R_{n,m}(x,y), \ x, y \in V_m$ the {\rm trace} (or the induced resistance) of $G_n$ on $V_m$. In particular, for $m=0$, we will use the notation $R_n(p, q), p, q \in V_0$ for simplicity. We also use $R(p, q)$ to denote the asymptotic geometric growth rate of $R_n(p, q)$ if it exists.
\end{definition}

\medskip

A function $h$ on $V_n$ is called {\it harmonic} on a subset $E\subset V_n$ if $h(x) = \sum_{x\sim y} c_n(x, y) h(y)$,  $x  \in E$. In the above, the function $\upsilon \in \ell(V_n)$ that attains the minimum (always exists) is a harmonic function on $V_n \setminus V_m$; we call it a {\it harmonic extension} of $u$ on $V_m$ to $V_n$.  As $\upsilon$ is harmonic on  the ``interior" of each subcell of $V_m$, we see  that  $\upsilon$  is a ``piecewise harmonic" function on $V_n$. These functions will be used to construct continuous functions in $B^\sigma_{2, \infty}$ as in the following.

\medskip

 \begin {proposition} \label {th2.6} For the primal energy $E_n[u], n\geq 0$ as defined in \eqref{eq1.3}, suppose  $\sigma$ satisfies $2\sigma > \alpha$,  and there exists an integer $N\geq1$ such that
\begin{equation} \label{eq2.5}
\rho^{-(2\sigma-\alpha)N}\leq R_N(p,q),  \quad  \forall \ p,q\in V_0, \ p \not = q,
\end{equation}
then $u \in \ell (V_0)$ has an extension to $K$, and consequently, $B^\sigma_{2, \infty}$ is dense in $C(K)$.
\end {proposition}

\medskip

\begin {proof} For any  $u \in \ell(V_0)$,   from the trace of $G_N$ on $V_0$, we have
{\small  \begin{align*}
& \min_{v \in \ell (V_N), v|_{V_0}= u}E_N[v]
=  \sum_{ p, q \in V_0} \frac{1}{R_N(p,q)}|u(p)-u(p)|^2,
\end{align*}}
where $E_N[v]$ is defined as in \eqref{eq1.3}. Multiplying $\rho^{-(2\sigma-\alpha)N}$ to both sides, and by \eqref{eq2.5}, we obtain
\begin{equation*}
 \min_{v \in \ell (V_n), v|_{V_0}= u}\Big\{ \rho^{-(2\sigma-\alpha)N} E_N[v]\Big \}
 \leq \sum\limits_{p,q \in V_0}\left|u(p)-u(q)\right|^2.
\end{equation*}
 With no confusion, we use $u$ again to denote the unique function in $\ell (V_N)$ that attains the minimum. By using this $u$ as initial data on each $F_\omega (V_0), |\omega| = N$, and continue this extension procedure to $V_{2N},V_{3N},\cdots$, there is $u$ on $V_* = \bigcup_{n\ge 0}V_n$ such that for all $k\geq1$,
\begin{equation*}
  \rho^{-(2\sigma-\alpha){kN}}E_{kN} [u]
 \leq\sum\limits_{p,q\in V_0}\left|u(p)-u(q)\right|^2.
\end{equation*}
 By Corollary \ref{th2.3} and  Proposition \ref{th2.4}, $u$ can be extended continuously to $K$ and $u\in B^\sigma_{2,\infty}$.

 \vspace {0.1cm}
It follows that for any  $v \in C(K) $,  if we let $v_n$ to be the restriction of $v$ on $V_n$,  we can  extend $v_n$ on each cell $K_{\omega}, |\omega|= n$ so that $v_n \in B^{\sigma}_{2, \infty}$ (this $v_n$ is a piecewise harmonic function). The sequence $\{v_n\}_{n=1}^\infty$ converges to $v$ uniformly.  This shows that $B^{\sigma}_{2, \infty}$ is dense in $C(K)$.
\end{proof}

\bigskip

To evaluate the trace $R_n(p,q)$ and
 estimate the energy functional on a network,  we will use some elementary techniques like the series law and parallel law  of resistance and the $\Delta$-Y transform. Recall the $\Delta$-Y transform \cite {K, S} states that
the $\Delta$-shaped resistors $(R_{12}, R_{23}, R_{31})$ and the $Y$-shaped resistors $(a, b,c)$  in Figure \ref{fig1} in any network are equivalent by the following relation
\begin{equation} \label {eq2.6}
  a =\frac{R_{12}R_{31}}{R}, \quad  b =\frac{R_{12}R_{23}}{R},\quad  c  =\frac{R_{31}R_{23}}{R},
\end{equation}
with $R= R_{12} +R_{23} +R_{31}$, and conversely,
\begin{equation} \label{eq2.7}
R_{12} = \frac rc, \quad R_{23} = \frac ra \quad R_{31} = \frac rb ,
\end{equation}
where $r = ab+ bc+ ca$.
\begin{figure}[h]
\textrm{\centering
\scalebox{0.15}[0.15]{\includegraphics{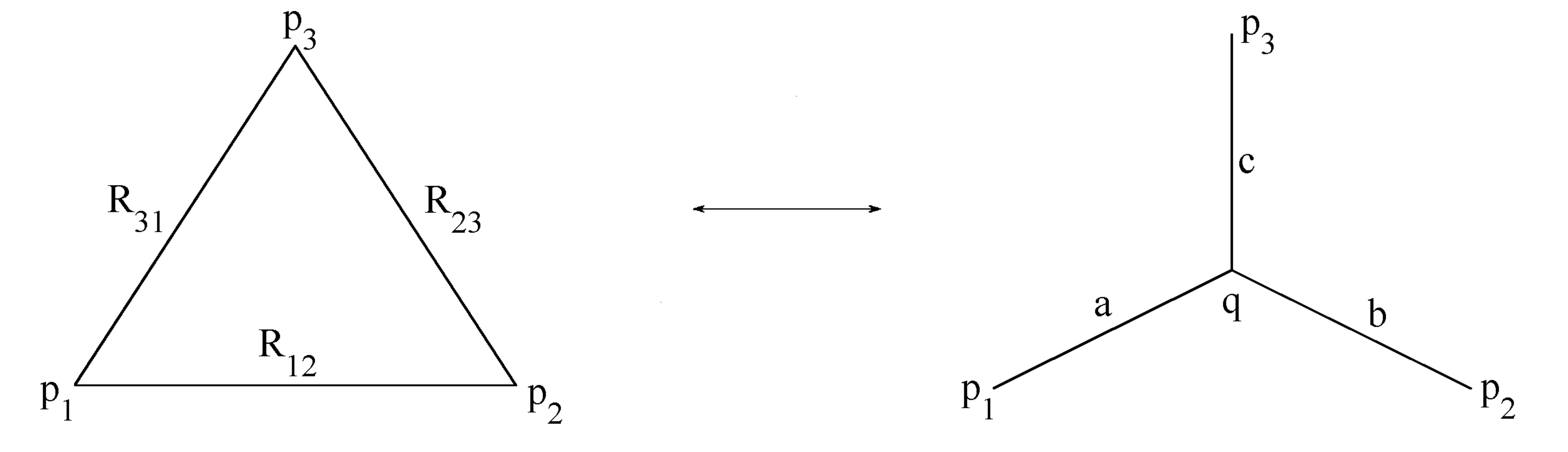}}\newline
}
\caption{$\Delta$-$Y$ transform}
\label{fig1}
\end{figure}

\medskip
In the  example of eyebolted Vicsek cross in Section 5, we need to use an electrical network with four terminals. We  give a version of equivalent electrical networks similar to the  $\Delta$-Y transform, and call it the{\it \begin{large}
${\boxtimes}$\end{large}-X transform}.

\medskip

\begin{lemma}\label{th2.7}
 For the two electrical networks as shown in Figure \ref{fig2} and assume that $yz = x^2$,  then they are equivalent  and the resistances satisfy
$$
 a= \frac {xy}{2(x+y)}\ , \quad \hbox {and} \quad  b = \frac {xz}{2(x+z)}  \ \Big (= \frac {x^2}{2(x+y)}\Big )\ ;
$$
equivalently,
$$
 x = 2(a+b)\ , \quad  y = \frac {2a}b(a + b)\ , \quad \hbox {and} \quad z = \frac {2b}a(a + b)\ .
$$
\end{lemma}

\begin{figure}[h]
\textrm{\centering
\scalebox{0.20}[0.20]{\includegraphics{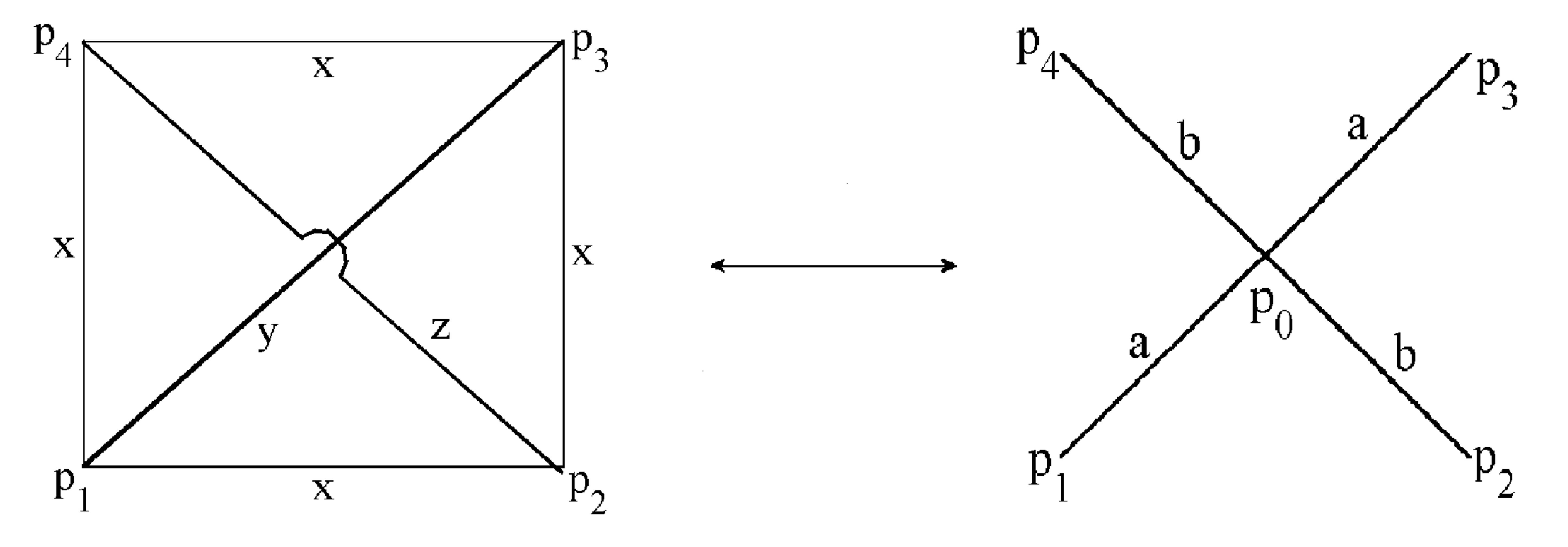}}\newline
}
\caption{equivalent networks with four vertices}
\label{fig2}
\end{figure}

\medskip

\begin {proof} We only outline the proof of the identity for $a$. By using the $\Delta$-Y transform  on the square together with $\overline {p_2p_4}$, it is easy to calculate the effective resistance of $p_1$ and $p_3$ is $x$ (this can also be obtained by observing that no current should pass through $\overline {p_2p_4}$). Then take this in parallel with the resistance $y$ on $\overline {p_1p_3}$, we have the desired expression.
\end{proof}

\bigskip
\section{\large\bf Quotient network}

\bigskip

 In this section, We will set up  the  equivalent relations and quotients on $V_n, n\geq 0$  which was first considered by Sabot \cite {Sa}. It will  be used to study the other critical exponent $\sigma^\#$  in Theorem \ref{th4.3}.

\bigskip

\begin {definition} \label {de3.1}
Let $\sim$ be an equivalence relation on $V_0$ that contains at least two equivalent classes.  We define the induced equivalent relation $\sim_n$ to be  the smallest equivalent relation on $V_n$ generated by

\medskip

\ \ (i) (embedding) for $x \sim_{n-1} y$ in $V_{n-1} (\subset V_n)$, then $x\sim_n y$ in $V_n$;

\medskip

\ (ii) (self-similar) for $x \sim_{n-1} y$ in $V_{n-1}$, then $F_i(x) \sim_n  F_i (y)$ for  $1\leq i\leq N$.

\medskip
\noindent We say that $\sim$ is a \textit{compatible (equivalence) relation}  if for any $n\geq0$ and any $x,y\in V_{n}$, $x\sim_n y$ in $V_n$ if and only if $x\sim_{n+1} y$  in $V_{n+1}$.
\end {definition}

\medskip

We will omit the subscript $n$ when there is no confusion,  and we write $V_0 = \bigcup_i J_i$ where the  $J_i$'s are  equivalent classes of $V_0$, and. Note that (ii)  implies
$$
F_\omega (p) \sim F_{\omega} (q), \quad  if \ \ p\sim q, \ p,q \in V_0;
$$
furthermore  if there are $q'\in V_0$, $|\omega'|=|\omega|$ such that $F_\omega (q) = F_{\omega'}(q')$, then for $p'\in V_0$ and $p'\sim q'$, then $F_\omega (p)\sim F_{\omega'}(p')$.
We will use $V_n^\sim, \ n\geq 0$ to denote the quotient spaces, i.e.,
$$
 V_n^\sim = \big \{ [F_\omega(J)]: \  J\in V_0^\sim, |\omega|=n\big \}.
 $$
 Here $[F_\omega(J)] $ is the union of the $F_{\omega'}(J'), J' \in  V_0^\sim,  |\omega'| \leq n$  where   $F_{\omega_i}(J_i) \cap F_{\omega_{i+1}} (J_{i+1}) \not = \emptyset $ for a finite sequence of cells in $V_m, m\leq n$ with  $F_\omega (J) = F_{\omega_1} (J_1), \cdots , F_{\omega_k} (J_k) = F_{\omega'} (J')$.

 \medskip

  In view of (i), the compatible condition is only imposed on the sufficiency.  It follows that for $m\leq n$,  $V_m^\sim$ can be identified as a subset of $V_n^\sim$.  The compatible relation  therefore induces an equivalence relation on $V_*$ (also denote by $\sim$)  with  $V_*^\sim= \bigcup_n V_n^\sim$ such that each equivalent class $J \in V_*$ is the union of  an increasing sequence of equivalent classes $J^{(n)}\in V_n$. It is easy to show inductively that if $J^{(n)}_1, J^{(n)}_2$ are distinct in $V^\sim_n$, then $J^{(m)}_1, J^{(m)}_2$ are distinct in $V^\sim_m$ for $m\geq n$, so that $J^*_1, J^*_2$ are distinct in $V_*$.

 \bigskip
We call an equivalent class $J$ of $V_n$ (or $V_*$)  a {\it boundary class} if $J\cap V_0\neq\emptyset$, and a {\it non-boundary class} otherwise.

 \bigskip

\noindent {\bf Examples}. For the Sierpinski gasket with $V_0 =\{p_1 , p_2, p_3\}$, the partition   $J_1= \{p_1, p_3\}$, $J_2 =\{p_2\}$ defines a compatible relation. It is easy to see that an element of $V_n^\sim$  is either a single vertex or is consisted of  consecutive vertices  on  a line segments parallel to $\overline {p_1p_3}$ (see Figure \ref{fig4}). There are two boundary classes in $V_*$,  $\{p_2\}$ and the set of dyadic points  on the line segment $\overline {p_1p_3}$.

\vspace {0.1cm}

Consider the pentagasket with $V_0 = \{p_i\}_{i=1}^5$  arranged in the counterclockwise direction. $J_i =\{p_i\}, i =1,2,3, \ J_4 = \{p_4, p_5\}$, then it defines a compatible relation. There are four boundary classes in $V_*$: the three singletons $J_i, i=1,2,3$, and $J^*_4$, which is a Cantor-set on the line segment $\overline {p_4 p_5}$.

On the pentagasket, if we let $J_1 = \{p_1, p_2\}, J_2 =\{ p_3, p_4, p_5\}$, then it is again a compatible relation, but the two boundary classes is more complicated  (see Figure \ref{fig3}),  its structure follows from Theorem \ref{th3.2}.

\bigskip

\noindent {\bf Remark}. In \cite {Sa}, Sabot first made use of the equivalent relation to study the Dirichlet form on a ramified self-similar set with a symmetric group $G$ acting on $V_0$. He defined a $G$-relation on $V_0$ by $x\sim y \Rightarrow gx \sim gy$ where $g \in G$, and extended  this to $V_n$ by rule (ii) and required it to be compatible  (it is called  preserved $G$-relation \cite [Section 4.2.1,  Definition 4.19] {Sa}).  This induce equivalent relation on $V_n$ is different from ours, which is generated by both rules (i) and (ii). The former definition is more limited, as it is easy to check that on the pentagasket, under rule (ii) only, then all non-trivial relation cannot be compatible.

\bigskip

In the following we will prove a theorem on the structure of the equivalent classes.
  For convenience, we call a set {\it equivalent set}  if it is consisted of equivalent elements. For $A, B$ equivalent sets, we write $A\sim B$ if there are $x\in A,
y\in B$ such that $x\sim y$.

\bigskip

\begin{theorem} \label {th3.2}
Let $V_0 = \bigcup_{i=1}^s J_i$ be the union of the equivalent classes of a compatible relation, and let $ \{J_i^*\}_{i=1}^s$ be the family of boundary classes in $V_*$. Then $\{\bar J_i^*\}_{i=1}^s$ are attractor of a graph directed system, and $\dim_H (\bar J_i^*) = \dim_B (\bar J_i^*)$.

\vspace {0.1cm}

 For the non-boundary classes, they  are finite unions of contractive similitude images of $\{J_i^*\}_{i=1}^s$.

\end{theorem}

 \medskip

 \begin {proof}
Let $V_0 = \bigcup_{i=1}^s J_i$, and let $J_i^{(n)}$ be the boundary  classes in $V_n, n\geq 0$ ($J_i^{(0)}= J_i$). Note that by (ii),  $F_k(J^{(0)}_j)$ is an equivalent set. It is easy to see that $J^{(1)}_i$ is generated by the  vertices in $ {\mathcal S}_i:=  \big\{F_k(J^{(0)}_i): F_k \ \hbox {has fixed point in } \ J^{(0)}_i \big \}$ (by (i)), together with those in $\big \{F_k(J^{(0)}_j):  F_k(J^{(0)}_j) \sim J,\ J \in {\mathcal S}_i,\  j \not =i \big \}$.  We can hence express $J_i^{(1)}$ as
$$
J_i^{(1)} = {\bigcup}_{j=1}^s{\bigcup}_{k\in \Gamma_{i, j}} F_k(J^{(0)}_j), \quad 1\leq i \leq s.
$$
where $\Gamma_{i, j} = \{k: F_k(J^{(0)}_j)\sim  J,\ J \in {\mathcal S}_i\}$ (see Figure \ref{fig3}b). We denote these boundary classes of $\{J_i^{(1)}\}_{i=1}^s$ by ${\mathcal B}^{(1)}$.

\vspace {0.1cm}

To determine the  non-boundary classes $\{I^{(1)}_i\}_{i=1}^t$ in $V_1$ (it may be empty), we  replace the above ${\mathcal S}_i$ by  $ {\mathcal T}_i:= \big \{F_k(J^{(0)}_i): F_k(J^{(0)}_i) \not \sim J, \ \forall \ J \in {\mathcal S}_j ,  \ 1\leq j\leq s \big \}$, and by the same way, we obtain
$$
I_i^{(1)} = {\bigcup}_{j=1}^s{\bigcup}_{k\in \Lambda_{i, j}} F_k(J^{(0)}_j),  \quad 1\leq i\leq t
$$
(see Figure \ref{fig3}b). Note that $I_i^{(1)} \subset V_1 \setminus V_0$, and the union of all the boundary and non-boundary classes is $V_1$. We denote this class $\{I^{(1)}_i\}_{i=1}^t$ by ${\mathcal N}^{(1)}_1$.

\medskip
Next, using the above procedure in $V_2$, we obtain,
$$
J_i^{(2)} = {\bigcup}_{j=1}^s{\bigcup}_{k\in \Gamma_{i, j}} F_k(J^{(1)}_j), \quad I_i^{(2)} = {\bigcup}_{j=1}^s{\bigcup}_{k\in \Lambda_{i, j}} F_k(J^{(1)}_j),
$$
denote the two classes by ${\mathcal B}^{(2)}$ and ${\mathcal N}^{(2)}_1$. There are other non-boundary classes  appeared, namely, those $F_k ({\mathcal N}^{(1)}_1),  1\leq k \leq N$  (see Figure \ref{fig3}c, note that for $I_i^{(1)} \in {\mathcal N}^{(1)}_1$, we have  $I_i^{(1)} \subset V_1 \setminus V_0$; hence $F_k (I_i^{(1)}) \subset F_k(V_1) \setminus F_k(V_0)$, so that $F_k (I_i^{(1)})$ remains an equivalent class in $V_2$).  We denote this  family  by ${\mathcal N}^{(2)}_2$. Hence  the family of non-boundary classes is  ${\mathcal N}^{(2)}_1 \cup {\mathcal N}^{(2)}_2$.

\medskip

Inductively,  we obtain ${\mathcal B}^{(n)}$,  and ${\mathcal N}^{(n)} := {\mathcal N}^{(n)}_1 \cup {\mathcal N}^{(n)}_2\cdots \cup {\mathcal N}^{(n)}_n$ where ${\mathcal N}^{(n)}_{\ell} = \{ F_k ({\mathcal N}^{(n-1)}_{\ell-1}):  1\leq k \leq N\} = \{ F_\omega ({\mathcal N}^{(n-\ell + 1)}_{1}):  |\omega| = \ell-1 \},\  2\leq \ell \leq n$. Therefore, for $ J_i^* \in {\mathcal B}^*$, i.e.,  $J_i^* \subset  V_*$, we have
$$
 J_i^* = {\bigcup}_{j=1}^s{\bigcup}_{k\in \Gamma_{i, j}} F_k(J^*_j),  \quad 1\leq i \leq s,
 $$
and $\Gamma_{i, j} = \{ k: F_k(J^*_j) \cap J_i^* \not = \emptyset\}$. We can set up the graph directed system $({\mathcal V}, \Gamma )$ with ${\mathcal V} = \{1, \cdots, s\}$ as the index of the  $ J^*_i$, and  $ \Gamma = \bigcup_{1\leq i, j\leq s} \Gamma_{ i,j}$ as the edge set. Therefore, $\bar J_i^*, 1\leq s \leq j$ are the attractors of $({\mathcal V}, \Gamma )$, and hence graph directed sets \cite {MW}.  For ${\mathcal N}^{*}$, we have for  $I_i^{*} \in {\mathcal N}^{*}_1$,
$$
I_i^{*} = {\bigcup}_{j=1}^s{\bigcup}_{k\in \Lambda_{i, j}} F_k(J^*_j),
$$
and for $\ell \geq 2$, ${\mathcal N}^*_\ell = \{ F_\omega ({\mathcal N}^{*}_1): |\omega| = \ell -1\}$. Hence for $I^* \in {\mathcal N}^{*}$, $\bar I^*$ is a finite union of graph directed sets.

\medskip
That $\dim_H (\bar J_i^*) = \dim_B (\bar J_i^*)$ is in \cite [p. 42] {F} where the graph directed system is assumed to be irreducible and the IFS is strongly separated.  In the present case, the IFS satisfies the OSC; the proof can easily be adjusted by observing that every ball of radius $r$ can intersect at most a number $\ell$ of cells of comparable size for some $\ell$. For the non-irreducible case,   Mauldin and Williams  in \cite [Theorems 4 and 5]{MW} studied the Hausdorff measures of the attractors (they can be infinite),  but left out the box dimension. This can easily be supplemented, and  for the sake of completeness and for the need of the box count in Lemma \ref{th3.4}, we will sketch the main idea of this in the Appendix.
\end{proof}

\begin{figure}[th]\label{fig3}
\textrm{
\begin{tabular}{cc}
\begin{minipage}[t]{1.6in} \includegraphics[width=1.6in]{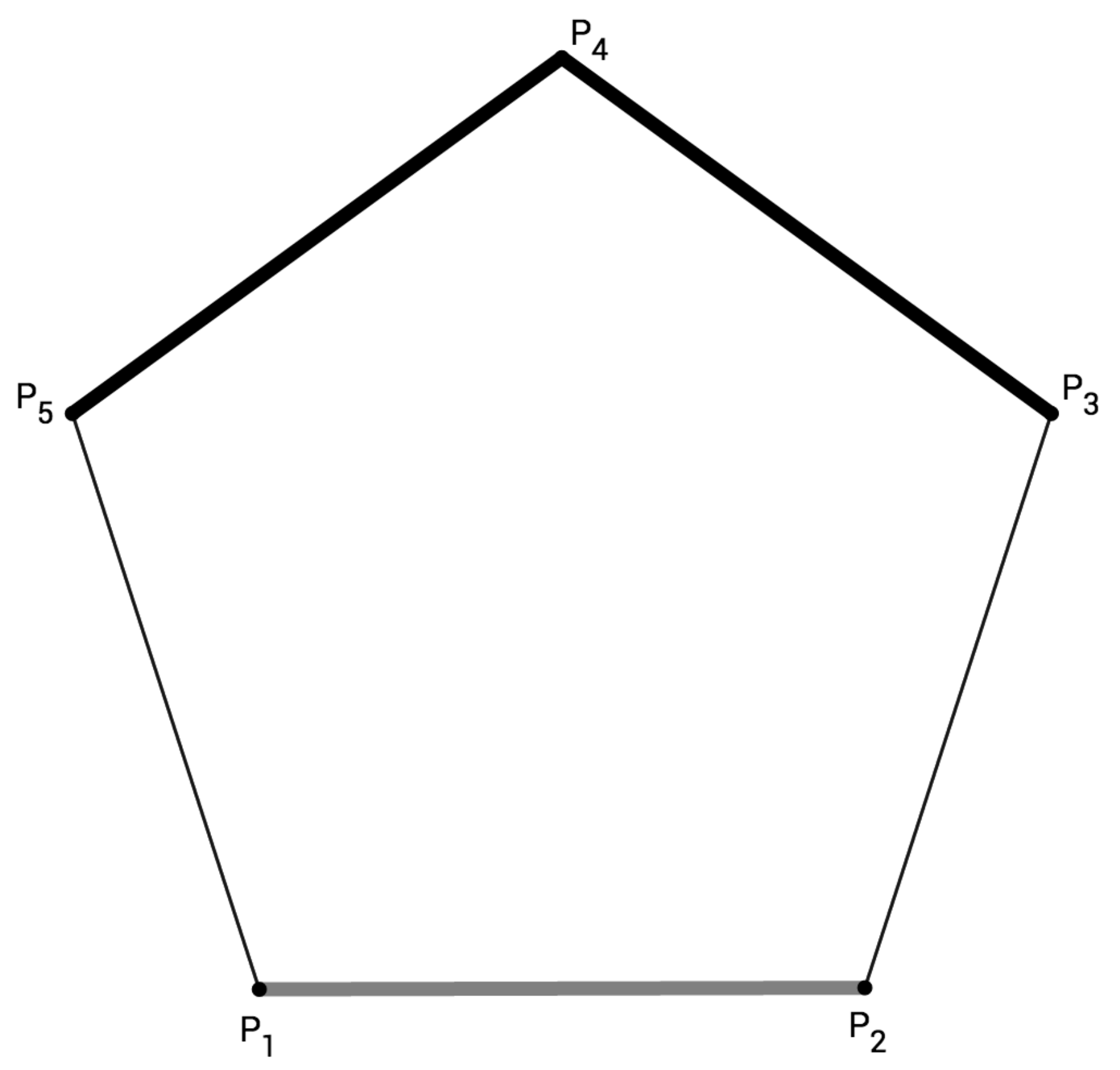}
 \center{3a}
 \end{minipage}  \qquad \quad
 \begin{minipage}[t]{1.6in} \includegraphics[width=1.6in]{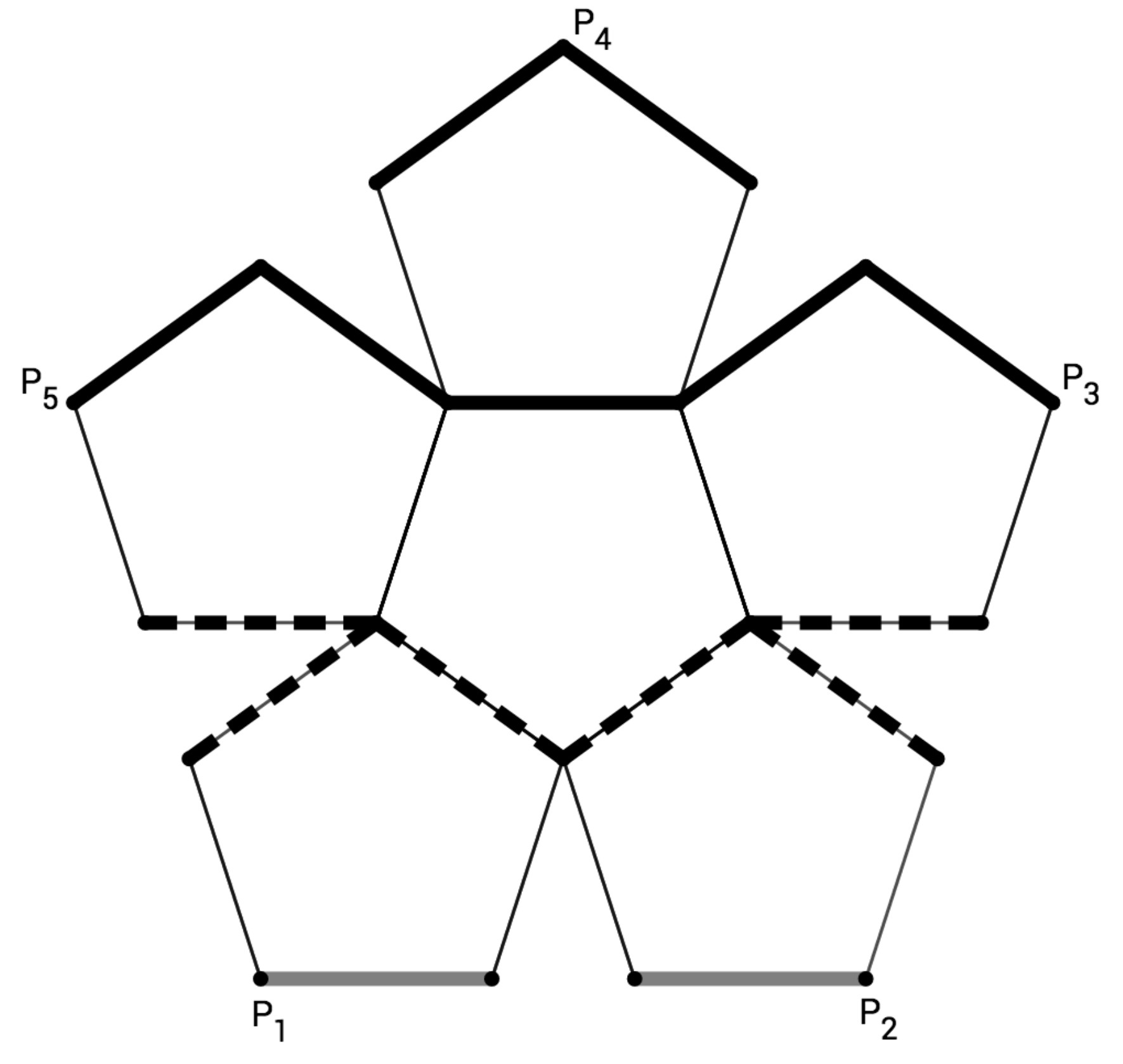}
 \center{3b}
\end{minipage}\qquad \quad
 \begin{minipage}[t]{1.6in} \includegraphics[width=1.6in]{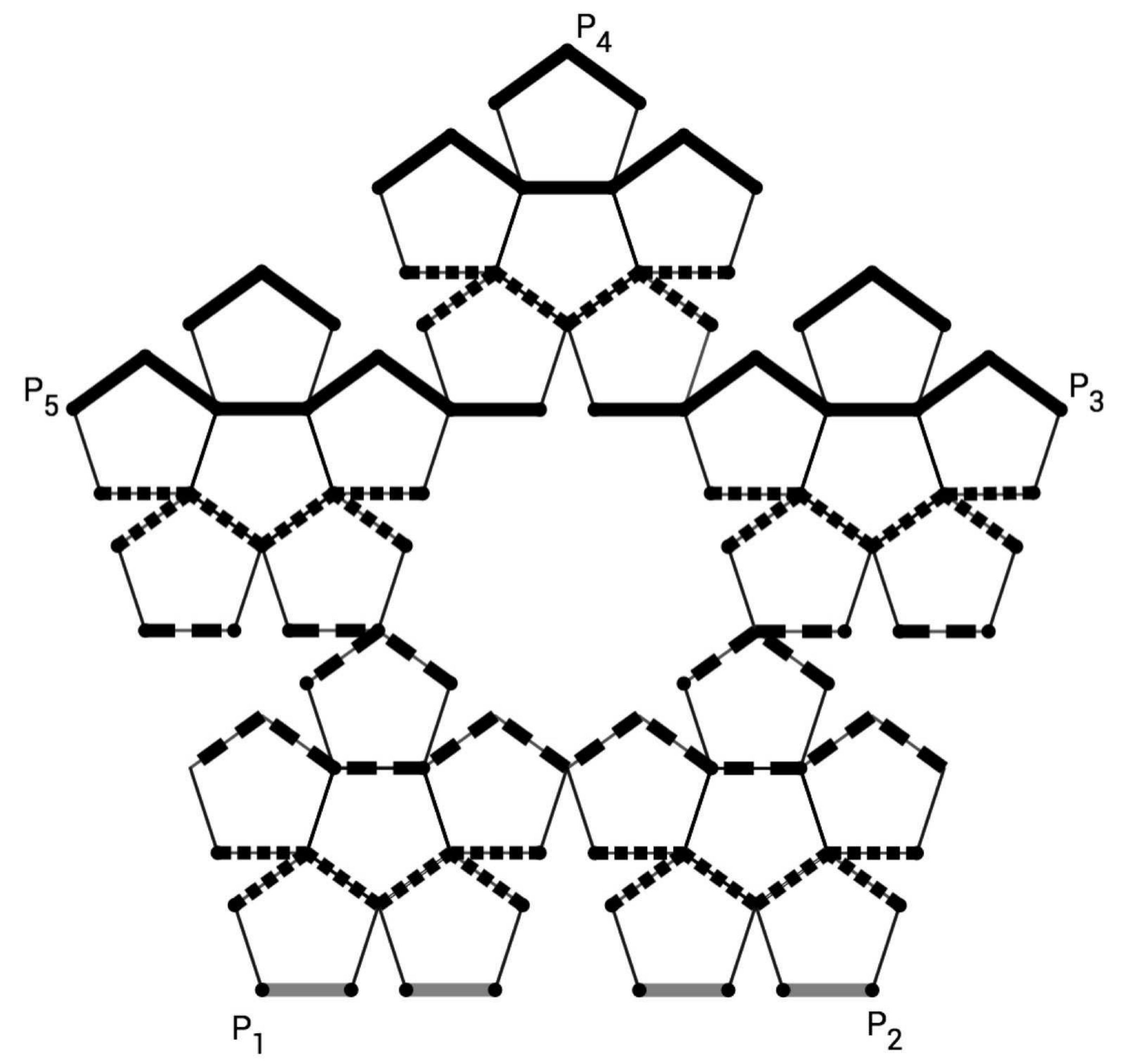}
 \center{3c}
\end{minipage} &
\end{tabular}
}\caption{ The equivalent classes defined by $J_1= \{p_1, p_2\}, J_2 = \{p_3, p_4, p_5\}$. In the iteration, the two boundary classes follows a graph directed system, and new non-boundary classes are generated in each step.}
\end{figure}

\bigskip

To obtain some separation property of the boundary classes, we introduce a  property on the compatible relation:

\vspace {0.2cm}
\noindent (B) \ {\it  Any  $1$-cell can intersect at most one boundary class in $V_1$.}

\medskip

Note that the previous examples all have property (B). On the other hand, on the pentagasket, if we let $J_1  =\{p_1\}$ and $J_2 = \{p_2, \cdots, p_5\}$. Then it is easy to see that it  defines a compatible relation, but does not have property (B); in this case,  $\bar J^*_2 = K$, and the non-boundary classes are some iterations of $\{p_1\}$.

\medskip

\begin {lemma} \label {th3.3}  Under property (B), any $n$-cell $V_\omega$, $n\geq 1$, intersects at most one boundary class in $V_n$.
\end{lemma}

\medskip

\begin{proof}  We use $\omega^-$ to denote the parent of $\omega$. Let $V_\omega$ be an $n$-cell, $n\geq 2$. For any boundary class $J$ in $V_n$, $J \cap V_{\omega^-}$ is a finite union of $F_{\omega^-}$-images of boundary classes in $V_0$ (by rules (i) and (ii) in Definition \ref{de3.1}), i.e., $F^{-1}_{\omega^-}(J \cap V_{\omega^-}) $ is a finite union of equivalent classes in $V_0$. As $F^{-1}_{\omega^-} (V_\omega)$ is a $1$-cell in $V_1$, it intersects at most one  $F^{-1}_{\omega^-}(J \cap V_{\omega^-})$ (by property (B)). Hence $V_\omega$  intersects at most one boundary class in $V_n$.
%
\end{proof}

\bigskip

Let
$
E_n[u] = \sum_{x,y\in V_\omega ,\ |\omega| = n} |u(x)-u(y)|^2
$
be the primal energy of $u$.   We will extend the consideration to the quotient network. For $u \in \ell (V_n^\sim)$, $u$ can be considered as a function in $ \ell(V_n)$ that takes constant value on each equivalent class $J \in V_n^\sim$. We have
\begin{equation}\label {eq3.1}
E^{\sim}_n [u] = \sum_{J, J' \in V_n^{\sim}}n_{J,J'}|u(J) - u(J')|^2= E_n[u], \quad  u\in\ell (V_n^\sim),
\end{equation}
where $n_{J,J'}$ is the number of the edges  connecting  $p\in J$ and $q\in J'$. (Note that $n_{J,J'}$  is well defined as the set of equivalent classes gives a partition of $V_n$,  hence for  $J$ as a subset in $V_n$, each edge going out $J$  will meet with another equivalent class.)
Let  $R^\sim_n(J, J')$ denote the corresponding trace of $V_n^{\sim}$ on $V_0^{\sim}$, then for $u \in \ell(V_0^\sim)$, as in \eqref{eq2.4'}, we have
{\small  \begin{align}\label{eq3.2}
& \min_{v\in\ell(V_n^\sim),\ v|_{V_0} = u }E_n[v]
=  \sum_{ i \not= j} \frac{1}{R^\sim_n(J_i,J_j)}|u(J_i)-u(J_j)|^2,
\end{align}}
We denote by $R^\sim(J_i,J_j)$ the asymptotic geometric growth rate of $\{R^\sim_n(J_i,J_j)\}_{n=0}^\infty$ if exists (see Definition \ref{de2.5}).

\bigskip
For $u\in \ell(V_*)$, and for a non-trivial equivalent class $J$ (i.e., contain more than one point)  of $V_n$, we denote by $E_{J,n}(u)$ the energy of $u$ on $J$:  the  summation of all the terms $(u(x)-u(y))^2$ in $E_n(u)$ with $x,y\in J\cap V_\omega,  |\omega|=n$. Similarly, we can define $E_{J,J', n}(u)$ with  $x \in J \cap V_\omega, y \in J'\cap V_\omega$. The following lemma is a sufficient condition of the existence of a renormalization factor localized at a boundary class.

\medskip

\begin {lemma} \label {th3.4}
 Let $J$ be a boundary class in $V_*$,  and $\bar J$ has  Hausdorff dimension $\gamma$. Suppose  there exists $p, q \in V_0,  p \not = q$ such that  $R:=R(p, q)  < \rho^{-(2- \gamma)}$,
then for any Lipschitz function $u$ on  $\bar J$, we have
\begin{equation} \label {eq3.3}
\sup_{n\geq 0} \ R^{n} E_{J, n} (u) < \infty
 \end{equation}

Furthermore if ${\mathcal H}^\gamma (\bar J) < \infty$, then the condition can be relax to $
R \leq \rho^{-(2- \gamma)}
$ and the same result holds.
\end {lemma}

\medskip
\begin{proof}     Let $N(\rho^n)$ be the count of the $n$-cells $K_\omega$ that intersects $\bar J$. It follows from Theorem \ref{th3.2} that the Hausdorff dimension of $\bar J$ equals its box dimension, and $\gamma = \lim\limits_{n\to \infty} \frac {\log N(\rho^n)}{\log \rho^n}$.   Hence for $\varepsilon >0$ satisfies $ R < \rho^{-(2- (\gamma + \varepsilon))}$, there exists $C>0$ such that  $ N(\rho^n) < C \rho^{-n(\gamma + \varepsilon)}$.  This implies that for all $n.$
$$
R^{n} E_{J, n} (u) \ \leq \ R^{n} \sum_{x, y \in J} |u(x)-u(y)|^2 \ \leq \ C' R^{n} \rho^{-n(\gamma + \varepsilon)} \rho^{2n} <C',
$$
and \eqref{eq3.3} follows.  If ${\mathcal H}^\gamma (\bar J) < \infty$, then we can actually have $ N(\rho^n) < C \rho^{-n\gamma}$ (see Appendix), and the above inequality holds for $
 R \leq \rho^{-(2- \gamma)}$.
\end{proof}

\bigskip
\section{\large\bf Critical exponents of Besov spaces}
In this section, we prove some general results for the critical exponents $\sigma^*$ and $\sigma^\#$ of the Besov spaces on the  p.c.f. sets with respect to the primal energy, then apply them to the two concrete cases  in the next section.

\bigskip

%
%

\medskip

\begin{theorem}\label{th4.1}
Let $K$ be a p.c.f. self-similar set with an IFS satisfying (\ref{eq2.1}). Assume $R(p, q)\ (>1), \ p, q \in V_0$ exist, and let
$$
R^*=\min\{R(p,q):\ p\neq q, \ p,q\in V_0\},
$$
then for the Besov spaces $B^\sigma_{2, \infty}$ defined on $K$,  the critical exponent $\sigma^*=\frac12\left(\frac{\log R^*}{-\log\rho}+\alpha\right)$.

\vspace {0.1cm}
Furthermore  at $\sigma^*$,

\vspace {0.1cm}
\ \ \ (i)\  if
$ R^{*n} \leq R_n(p, q)$ \ for all  $n\geq 0$ and $ p, q \in V_0$,
then $B^{\sigma^*}_{2, \infty}$ is dense in $C(K)$;

\vspace {0.1cm}

\ (ii) \  if  $R^*$ satisfies \
$
\lim\limits_{n\to \infty}\  \frac {R^{*n}}{R_n(p, q)} = \infty$ \ for some $p, q \in V_0$, then $B^{\sigma^*}_{2, \infty}$ is not dense in $C(K)$.

\end{theorem}

\medskip

\begin{proof}   Let $\sigma = \frac12\left(\frac{\log R^*}{-\log\rho}+\alpha\right) +2\varepsilon$, $\varepsilon>0$, then $2\sigma - \alpha>0$.  We will  prove that there exist $p, q \in V_0$, $p \not = q$,  such that $u(p) = u(q)$ for  any $u \in B^\sigma_{2, \infty}$. This implies $B^\sigma_{2, \infty}$ is not dense in $C(K)$, and by definition, we have  $\sigma^* \leq \frac12\left(\frac{\log R^*}{-\log\rho}+\alpha\right)$.

\vspace {0.1cm}
To this end, for any $u \in B^\sigma_{2, \infty}$, we  restrict $u \in \ell(V_0)$ with values $u(p_i), \  p_i \in V_0$.  From the trace of $G_n$ on $V_0$, we obtain
{\small  \begin{align}\label{eq4.1}
& \min_{v \in \ell (V_n), v|_{V_0}= u}E_n[v]
=  \sum_{ i \not= j} \frac{1}{R_n(p_i,p_j)}|u(p_i)-u(p_j)|^2.
\end{align}}
Let $R_\varepsilon^* = \rho^{-2\varepsilon}R^*$. Multiplying $R_\varepsilon^{*n}$ to both sides of (\ref{eq4.1}), it reduces to
{\small \begin{equation} \label{eq4.2}
\min_{v \in \ell (V_n), v|_{V_0}= u}\Big\{\rho^{-\big(\frac{\log R^*}{-\log\rho}+2\varepsilon\big)n}E_n[v]\Big\}
 =
\sum_{ i \not= j} \frac{R_\varepsilon^{*n}}{R_n(p_i,p_j)}|u(p_i)-u(p_j)|^2.
\end{equation}}
By the definition of $R^*$, we have  $R^*=R(p,q)$ for some  $p,q\in V_0$, and
\begin{equation} \label {eq4.3}
{\small \frac{R_\varepsilon^{*n}}{R_{n}(p,q)}}\geq \rho^{-\varepsilon n} \rightarrow\infty\quad \text{ as } n\rightarrow\infty.
\end{equation}
 As $u \in B^\sigma_{2, \infty}$, the left hand side of \eqref{eq4.2} is uniformly bounded for all $n>0$ (by Proposition \ref{th1.1}). Hence by \eqref{eq4.3}, $u(p)=u(q)$ on the right hand side of (\ref{eq4.2}). This completes the proof of  $\sigma^* \leq \frac12\left(\frac{\log R^*}{-\log\rho}+\alpha\right)$.

\vspace{0.2cm}

 Next we consider $\sigma <\frac12\left(\frac{\log R^*}{-\log\rho}+\alpha\right)$ such that $2\sigma - \alpha >0$ .  By the definition of $R^*$, we can find an  $N=N(\sigma)$ such that
\begin{equation} \label{eq4.4}
\rho^{-(2\sigma-\alpha)N}\leq R_N(p,q),  \quad  \forall \ p,q\in V_0, \ p \not = q.
\end{equation}
Then by Proposition \ref{th2.6},   $B^{\sigma}_{2, \infty}$ is dense in $C(K)$. Therefore $\frac12\left(\frac{\log R^*}{-\log\rho}+\alpha\right) \leq \sigma ^*$.

\vspace {0.2cm}

For (i) in the last part, we have
$$
\rho^{-(2\sigma^*-\alpha)n } = R^{*n} \leq R_n(p, q) \quad  \forall \ p,q \in V_0,
$$
and the assertion follows from  Proposition \ref{th2.6}.  For (ii), if we replace \eqref{eq4.3}  there by the assumption,  then the same argument applies  and we  conclude that  $u(p) = u(q)$, so that $B^{\sigma^*}_{2, \infty}$ is not dense in $C(K)$.
\end{proof}

\bigskip

  It is important to know the density of $B^{\sigma^*}_{2, \infty}$ in $C(K)$. Part (i) covers the standard cases that  renormalization factors exist \cite {K}; there is example in Section 4 (eyebolted Vicsek cross) that satisfies (ii).  In addition, we have a situation there (Sierpinski sickle) that ${R^{*n}}\asymp {R_n(p, q)}$, and the density question is not covered in the above two situations. We make use the technique of quotient network developed in the last section to handle this case.

\bigskip
\begin {proposition} \label {th4.2} With the assumptions in Theorem \ref{th4.1}, suppose there is a compatible relation  with property (B), and satisfies

\vspace{0.1cm}

(i) For each  non-trivial boundary  class $J$ of $V_*$  and for any $u$ on $J\cap V_0$, there is an
 extension of $u$ on $J$ such that \ $\sup_{n\geq0}R^{*n}E_{J,n}(u)<\infty$.

\vspace {0.1cm}

 (ii) for any two distinct equivalence classes $J_i,J_j$ of $V_0$, $R^\sim(J_i,J_j)$ exists and satisfies $R^{\sim}(J_i,J_j)>R^*$.

\vspace {0.1cm}
\noindent  Then $B^{\sigma^*}_{2, \infty}$ is dense in $C(K)$.

\end{proposition}

\medskip
\noindent {\bf Remark}. A sufficient condition for the existence of the $u$ in condition (i) is provided in Lemma \ref{th3.4}.  The main idea behind condition (ii) is to use the equivalent class (i.e., shorting in the sense of electrical network) to get rid of the small $R_n(p,q)$,  and reduces to an expression analogous to Theorem \ref{th4.1}(i). An example of this is the Sierpinski sickle in Section 4.

\bigskip

\begin{proof}
We let ${\mathcal B}_n$ (${\mathcal B}_*$) denote the family of  boundary classes in $V_n$ ($V_*$ respectively), and let $B_n$ ($B_*$ respectively) be the union of the boundary classes. For $u\in \ell (V_0)$, we show that  $u$ can be extended to $V_*$, then to $K$, and has bounded $B^{\sigma^*}_{2, \infty}$-norm.
This will imply $B^{\sigma^*}_{2, \infty}$ is dense in $C(K)$ by  a similar argument (piecewise extension) as in the last paragraph of Proposition \ref{th2.6}.  We will prove the assumption in three steps.  (See Figure \ref{fig4}, note that the SG there is just for simplicity for illustration, it does not satisfies condition (ii))

\medskip

{\bf Step 1}. We  use (i) to extend $u$ to the  $J \in {\mathcal B}_*$  such that
$\sup_{J\in {\mathcal B}_*}\sup_{n\geq 0 }R^{*n}E_{J,n}(u)  \leq M$ for some $M>0$.  Hence $u$ is  defined on $B_* \subset V_*$. We also define $u$ on $V_1\setminus (V_0\cup B_1)$ to be $\min_{z\in V_0} u(z)$.

\medskip
 Our main task is to extend $u$ to the rest of $V_*$ and has bounded energy. We call $V_\omega$ a boundary cell if   $V_\omega \cap B_* \not = \emptyset$, and a non-boundary cell otherwise.

\vspace {0.1cm}

{\bf Step 2}.  We use induction  on  $|\omega|\geq 2$ to define $u$ on the boundary cells $V_\omega$ as well as  $V_{\omega i}, 1\leq i \leq  N$,  so that $u$ is  constant on the  equivalent classes. For this, suppose we have defined  such $u$ on boundary cells $V_\omega$, $|\omega|=n-1$. We carry out the induction in two steps. For a non-boundary cell $V_{\omega i}$ that has an equivalent class $I$ intersects $V_\omega$, we assign $u$ on $I$ to take the value $I \cap V_{\omega}$. Next  by property (B),  $V_{\omega} ( \subset V_n)$ intersects a unique boundary class, say $J$ (Lemma \ref{th3.3}). We assign $u$  to be $\min_{z\in V_{\omega}\cap J} u(z)$ on the rest of the  vertices in $V_{\omega i}$, $1\leq i \leq N$ that are not yet defined.

\vspace {0.1cm}

{\bf Step 3}.  Let
$
{\mathcal N}= \{\omega: \ V_\omega \hbox { is non-boundary cell},   V_{\omega^-}  \hbox {
\ is boundary cell} \},
 $
 (See Figure \ref{fig4} for the dotted cells.) In Step 2, we have already defined the values of $u$ on these cells. To complete the construction, we continue to define $u$ on $V_* \cap K_\omega$. By using assumption (ii) and (\ref{eq3.2}) on $V_{\omega}^\sim$, for small enough $\varepsilon>0$ such that $R^*<\min_{i\neq j} R^{\sim}(J_i,J_j)-\varepsilon$, we can choose $N$ big enough so that the harmonic extension of $u$ from $V_{\omega}$ to $V_{n+N}\cap K_{\omega}$ satisfies $$
  \big({\min}_{i\neq j} R^{\sim}(J_i,J_j)-\varepsilon\big)^N E_N(u\circ F_{\omega})\leq E_0(u\circ F_{\omega}),
   $$
and $u$ is constant on the equivalent classes in $V_{\omega + N}$.  By repeating the harmonic extension of $u$ on $V_{n+2N}\cap K_{\omega},V_{n+3N}\cap K_{\omega},\cdots$, we have $\sup\limits_{k\geq0}\Big(\min\limits_{i\neq j} R^{\sim}(J_i,J_j)-\varepsilon\Big)^{kN}
   E_{kN}(u\circ F_{\omega})\leq E_0(u\circ F_{\omega})$, and consequently by Corollary \ref{th2.3},
\begin{equation}\label{eq4.5}
\sup\limits_{k\geq0}\Big(\min\limits_{i\neq j} R^{\sim}(J_i,J_j)-\varepsilon\Big)^k E_k(u\circ F_{\omega})\leq CE_0(u\circ F_{\omega}).
\end{equation}
This completes the construction of $u$ on $V_*$.
\begin{figure}[h]\label{fig4}
\centering
\textrm{
\scalebox{0.25}[0.25]{\includegraphics{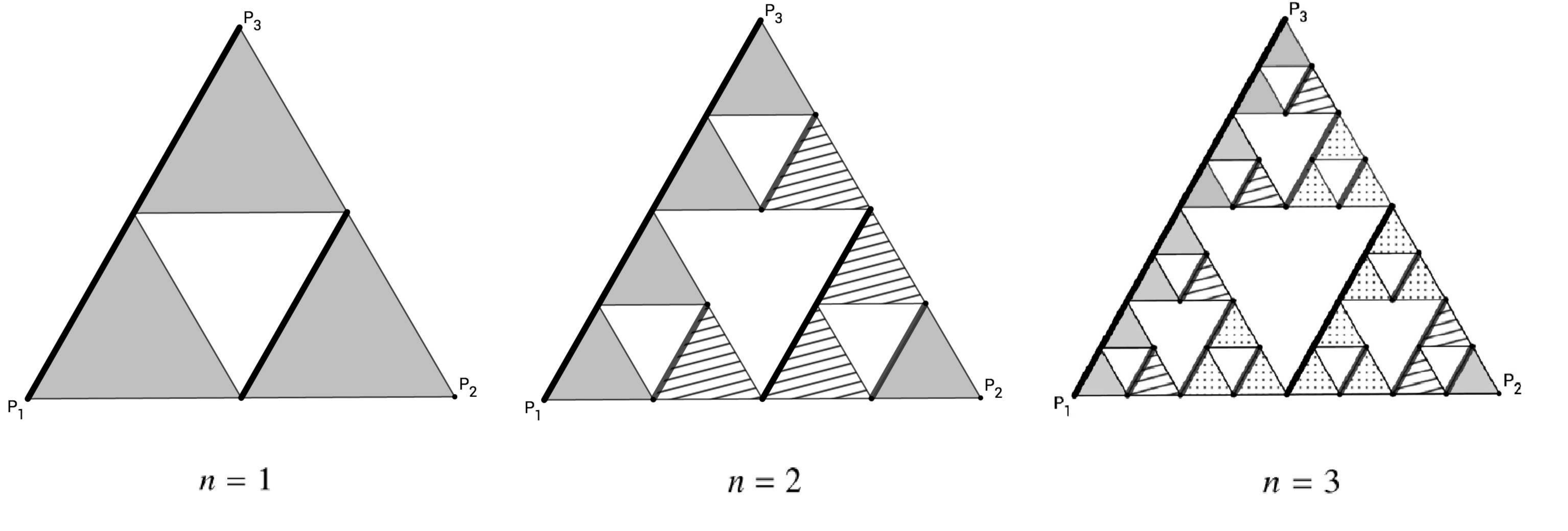}}
}
\caption{ Let $J_1 =\{p_1, p_3\}, J_2 = \{p_2\}$ on $V_0$. For $n = 1,2,3$, the  vertices in heavy line segments are the equivalent classes, and the grey cells are the boundary cells, Step 2 determines $u$ on the new vertices of the grey cells  and the lined cells; Step 3 determines $u$ in the dotted cells.}
\end{figure}

Finally we  estimate the energy $E_n(u)$. Note that $\sum_{J \not \in {\mathcal B}_n}E_{J, n} =0$ as $u$ is constant on such $J$ by construction; also for $n\geq1$, $\sum_{J,J' \in {\mathcal B}_n, J \not = J'} E_{J, J',n} =0$ because elements in $J$ and $J'$ have zero conductance (by Lemma \ref{th3.3}). Hence we can write  $E_n(u)$ as
{\small \begin{equation}  \label {eq4.6}
E_n(u)=\sum\limits_{J\in {\mathcal B}_n }E_{J,n}(u)+\sum\limits_{J\in {\mathcal B}_n,  J'\not \in {\mathcal B}_n}E_{J,J',n}(u)+\sum\limits_{J \not \in {\mathcal B}_n,  J' \not \in {\mathcal B}_n}E_{J,J',n}(u).
\end{equation}}
For the second sum, we have
{\small \begin{align}
\sum\limits_{J\in {\mathcal B}_n, J'\not \in {\mathcal B}_n}E_{J,J',n}(u)
& = \sum\limits_{|\omega|=n,V_\omega\cap J\neq\emptyset, J\in {\mathcal B}_n }\ \sum\limits_{x\in V_\omega\cap J}\ \sum\limits_{y\in V_\omega\cap J', J'\not \in {\mathcal B}_n}(u(x)-u(y))^2.\notag\\
&\leq \ C\sum\limits_{\ |\omega|=n,V_\omega\cap J\neq\emptyset, J\in {\mathcal B}_n}\sum\limits_{x\in V_\omega\cap J}\left(u(x)-u(x')\right)^2 \notag \\
 &\leq \ C'\sum\limits_{\ |\omega|=n,V_\omega\cap J\neq\emptyset, J\in {\mathcal B}_n}\left(\sum\limits_{x,z\in V_\omega\cap J}\left(u(x)-u(z)\right)^2+\sum\limits_{x,z\in V_{\omega^-}\cap J}\left(u(x)-u(z)\right)^2\right)\label{eq4.7}\\
&\leq \ C''\sum\limits_{ J\in {\mathcal B}_n}E_{J,n}(u)+E_{J,{n-1}}(u),\notag
\end{align}}
where $u(y)$ takes value $ {\min}_{z\in V_{\omega^{-}} \cap J} u(z) =u(x')$ for some $x' \in V_{\omega^{-}} \cap J$; the second inequality follows by choosing a path on $J\cap V_{n}\cap K_{\omega^{-}}$ from $x$ to some point $z$ in $J\cap V_{\omega^-}$, this is controlled by the first term in \eqref{eq4.7}, and then connecting $z$ and $x'$ in $J\cap V_{\omega^{-}}$ which is controlled by the second term in \eqref{eq4.7}.

\vspace {0.1cm}

For the last sum in \eqref{eq4.6},  as $J, J' \not \in {\mathcal B}_n$, by the construction, $u$ is constant on $J$ and on $J'$.  It follows from \eqref{eq4.5}  that
{\small \begin{align}
\sum\limits_{J \not \in {\mathcal B}_n, J' \not \in {\mathcal B}_n}E_{J,J',n}(u)\ &=\ \sum\limits_{J \not \in {\mathcal B}_n, J' \not \in {\mathcal B}_n}E^\sim_{J,J',n}(u)=\ \sum\limits_{\omega\in {\mathcal N},|\omega| =k\leq n}E_{n-k}(u\circ F_{\omega})\notag\\
&\leq C\sum\limits_{\omega\in {\mathcal N},|\omega|=k\leq n}\Big(\frac{1}{\min\limits_{i\neq j} R^{\sim}(J_i,J_j)-\varepsilon}\Big)^{n-k}E_0(u\circ F_{\omega}).\label{eq4.8}
\end{align}}
For the  $\omega$ in the sum,  $V_{\omega^-}\cap J\neq \emptyset$ for some boundary class $J$,
then by the construction, the values of $u$ on $V_k\cap K_{\omega^-}$ are defined from the values of $u$ on $J\cap K_{\omega^{--}}$, and hence
\begin{equation*}
E_0(u\circ F_{\omega})\leq E_1(u\circ F_{\omega^-})\leq CE_{k,J\cap K_{\omega^{--}}}(u).
\end{equation*}
Thus (\ref{eq4.8}) is smaller than
\begin{equation*}
C'\sum\limits_{J \in {\mathcal B}^*}\sum\limits_{k=0}^n\Big(\frac{1}{\min\limits_{i\neq j} R^{\sim}(J_i,J_j)-\varepsilon}\Big )^{n-k}E_{k,J}(u).
\end{equation*}

 To conclude, the above two estimates and (ii) imply that \eqref {eq4.6} satisfies
{\small \begin{equation*}
R^{*n}E_n(u)\leq C\sup\limits_{0\leq k\leq n}R^{*k}E_{J,k}(u)+C'\sum\limits_{J \in {\mathcal B}^*}\sum\limits_{k=0}^nR^{*k}E_{k,J}(u)+E_0(u)<\infty.
\end{equation*}}
Therefore $R^{*n}E_n(u)$ is uniformly bounded on $n$.   This implies that  $u\in B^{\sigma^*}_{2, \infty}$ and completes the proof.
\end{proof}

\bigskip

In the following, we will consider the second critical exponent $\sigma^\#$.

\medskip

\begin{theorem}\label{th4.3}
Let $K$ be a p.c.f. self-similar set with an IFS satisfying (\ref{eq2.1}). Assume $R(p, q)\ (>1), \ p, q \in V_0$ exist. Let
\begin{align*}
R^\#=\min\big \{s:\ & \text{ $\forall$  $p\not = q$ in $V_0$, $\exists$ a chain $p=p_1,p_2,\cdots,p_m=q$ in $V_0$}\\
               & \text{ $\ni$ $R(p_i,p_{i+1})\leq s$,  $1\leq i\leq m-1$}\big \}.
\end{align*}
Suppose there is a compatible relation  on $V_0$ such that
for any small $\varepsilon >0$,
\begin{equation}\label{eq4.9}
\lim_{n\to \infty} \ \frac{(R^\#)^{(1-\varepsilon )n}} {R_n^\sim(J_i, J_j)}= 0,   \quad \forall \ J_i, J_j \in V_0^{\sim},  \  J_i \not = J_j.
\end{equation}
Then $\sigma^\#=\frac12\left(\frac{\log R^\#}{-\log\rho}+\alpha\right)$\ .
\end{theorem}

\medskip
\begin{proof}  To prove $\sigma^\# \leq \frac12\left(\frac{\log R^\#}{-\log\rho}+\alpha\right)$, we consider $\sigma = \frac12\left(\frac{\log R^\#}{-\log\rho}+\alpha\right) +2\varepsilon$, we claim that for $u \in B^\sigma_{2, \infty}$, $u(p) = u(q)$ for all $p, q \in V_0$. Then the same argument apply to any $u\in \ell(V_n)$, and hence $u$ is a constant function.

\vspace {0.1cm}

For $\varepsilon >0$ and for $R_\varepsilon^{\#} = R^\# \cdot\rho^{-2 \varepsilon}$,
the definition of $R^\#$ implies that  there is a chain $p=p_1,p_2,\cdots,p_m=q$ such that
 $R(p_i,p_{i+1})\leq R^\#$ for $1\leq i\leq m-1$. Hence
$$
\frac{R^{\#n}_\varepsilon}{R_{n}(p_i,p_{i+1})}\geq \rho^{-\varepsilon n} \rightarrow\infty\text{ as } n\rightarrow\infty. \quad \forall \ 1\leq i\leq m-1.
$$
 By the same argument as in Theorem \ref{th4.1}, we conclude that $u(p_i) = u(p_{i+1})$ for all $1\leq i\leq m-1$, so that $u(p)= u(q)$, and the claim follows.

\vspace {0.2cm}
Next we consider $\sigma <\frac12\left(\frac{\log R^\#}{-\log\rho}+\alpha\right)$ such that $2\sigma -\alpha>0$.  We show that each $u \in \ell(V_0^\sim)$ (equivalently, $u \in \ell(V_0)$ and  is constant on each $J_i$) can be extended to be in $B^\sigma_{2, \infty}$. Then  $B^\sigma_{2, \infty}$ contains non-constant function, and by definition,   $\frac12\left(\frac{\log R^\#}{-\log\rho}+\alpha\right) \leq \sigma ^\#$.

 \vspace {0.1cm}
 The proof of the statement is the same as the corresponding part in Theorem \ref{th4.1}.  By  (\ref{eq4.9}), we can find an  $N=N(\sigma)$ such that for all distinct $i,j$,
\begin{equation}\label{eq4.10}
\rho^{-(2\sigma-\alpha)N}\leq R^\sim_N(J_i,J_j), \quad  \forall \ J_i, J_j \in V_0^\sim, \ \ J_i \not = J_j,
\end{equation}
then by \eqref{eq3.2} and Proposition \ref{th2.6}, we obtain $u \in B^\sigma_{2, \infty}$.
\end{proof}

\medskip

\noindent {\bf Remark.} It follows from the above argument that if  $ u \in \ell (V_n^\sim)$, then $u$ can be extended to be in $B^\sigma_{2, \infty}$, $\sigma < \sigma^\#$ and $u$ is constant on each equivalent class $J \in V_*^\sim$.

\bigskip

\begin{proposition}\label{th4.4}  With the same assumption as in Theorem \ref{th4.3}, Suppose that
\begin{equation}\label{eq4.11}
\max\{\dim_H \bar J :\ J \hbox {\rm{ a boundary class}}  \}< \dim_H K.
\end{equation}
 Then for $\sigma < \sigma^\#$, $B^\sigma_{2, \infty}$ is dense in $L^2(K,\mu)$. If further, for any $u$ on $\ell(V_0^\sim)$, there is an extension of $u$ on $V_*^\sim$ such that \
\begin{equation}\label{eq4.12}
\sup_{n\geq0}R^{\#n}E_{n}(u)<\infty.
\end{equation}
Then $B^{\sigma^\#}_{2, \infty}$ is dense in $L^2(K,\mu)$.
\end{proposition}

\medskip
\begin{proof} For any $\varepsilon >0$ and for any  $f\in L^2(K, \mu)$, let $g\in C(K)$ satisfying $||g-f||_2 \leq \varepsilon$.
Let $\delta>0$ be such that for  $x,y\in K$ with $|x-y|\leq\delta$, $|g(x)-g(y)|\leq\varepsilon$. For $\delta >0$,  we let
$$
Q_{\delta,n}={\bigcup}_{|\omega|=n}\{ K_\omega: \ K_\omega \cap J \not = \emptyset,\  {\rm diam} (J) \geq \delta ,  \ J \in V_*^\sim \},
$$
and let $Q_\delta = \bigcap_{n=1}^\infty Q_{\delta,n}$. Then $Q_\delta$ is union of the closures of finitely many equivalent classes. In view of Theorem \ref{th3.2}, $\dim_H(Q_\delta)$ is the maximum of  the Hausdorff dimensions of the boundary classes. By \eqref{eq4.11}, $\mu (Q_\delta) =0$,
 we can find $N: =N(\varepsilon,\delta)$ satisfies $\mu(Q_{\delta, N}) < \varepsilon$.  Define $g_N$ on $V_N$ such that
\begin{equation*}
g_N(x) =
\sup_{y\in V_N\cap J}g(y), \quad \hbox {for}  \ x \in V_N\cap J, \ \ J \in V_*^\sim.
\end{equation*}
Then $g_N\in \ell(V_N^\sim)$. In view of the above Remark, we can extend $g_N$ to  $B^\sigma_{2, \infty}$,    which  is continuous and is constant on each equivalent class of $V_*^\sim$,  still denote this extension by $g_N$.  Note  that  $g_N$ takes maximum and minimum values on $V_N$.
It follows that
  \begin{align*}
   ||g_N-g||^2_2
   &=\int_{Q_{\delta,N}}|g_N-g|^2d\mu+\int_{K\setminus Q_{\delta,N}}|g_N-g|^2d\mu \\
   & \leq \left(2||g||_\infty\right)^2\cdot\mu(Q_{\delta,N})+(2\varepsilon)^2\cdot\mu(K)
   \ \leq\ c\varepsilon^2
\end{align*}
for some $c>0$ (independent of $g$). Hence $||g_N-f||_2  \leq  ||g_N-g||_2+ ||g-f||_2 \leq  (c+1)\varepsilon$. The denseness of $B^{\sigma}_{2, \infty}$ in $L^2(K, \mu)$ follows.

\medskip
 For the second part, the  assumption \eqref{eq4.12}  on $V_*^\sim$ implies that we can extend any $u\in \ell(V_0^{\sim})$ to be in $\ell(V_*^{\sim})$ and in $B^{\sigma^\#}_{2, \infty}$, so that the same argument shows that
$B^{\sigma^\#}_{2, \infty}$ is dense in $L^2(K, \mu)$.
\end{proof}

\bigskip
\bigskip

\section{\large\bf P.c.f. sets with inhomogeneous traces}
\bigskip

In this section, we will construct and study the two asymmetric p.c.f. sets as announced. We will assume  the primal energy on $V_n$, i.e.,   $E_n [u] = \sum_{x,y\in V_\omega ,\ |\omega| = n} |u(x)-u(y)|^2$.

\bigskip
 {\bf 1.  Eyebolted  Vicsek cross}.
In ${\Bbb R}^2$, let $\{p_1,p_2,p_3,p_4\}$ be the four vertices of the unit square $S$, and let $p_0$ be the center of $S$, that is,
$p_0=(0,0)$ and $p_1= (-1/2,-1/2), \ p_2 = (1/2,-1/2), \ p_3 = (1/2,1/2),\ p_4=(-1/2,1/2)$.
Divide $S$ into a mesh of sub-squares of size $1/9$, and pick $21$ sub-squares as shown in Figure \ref{fig5}.

\begin{figure}[th]
\textrm{
\begin{tabular}{cc}
\begin{minipage}[t]{1.7in} \includegraphics[width=1.7in]{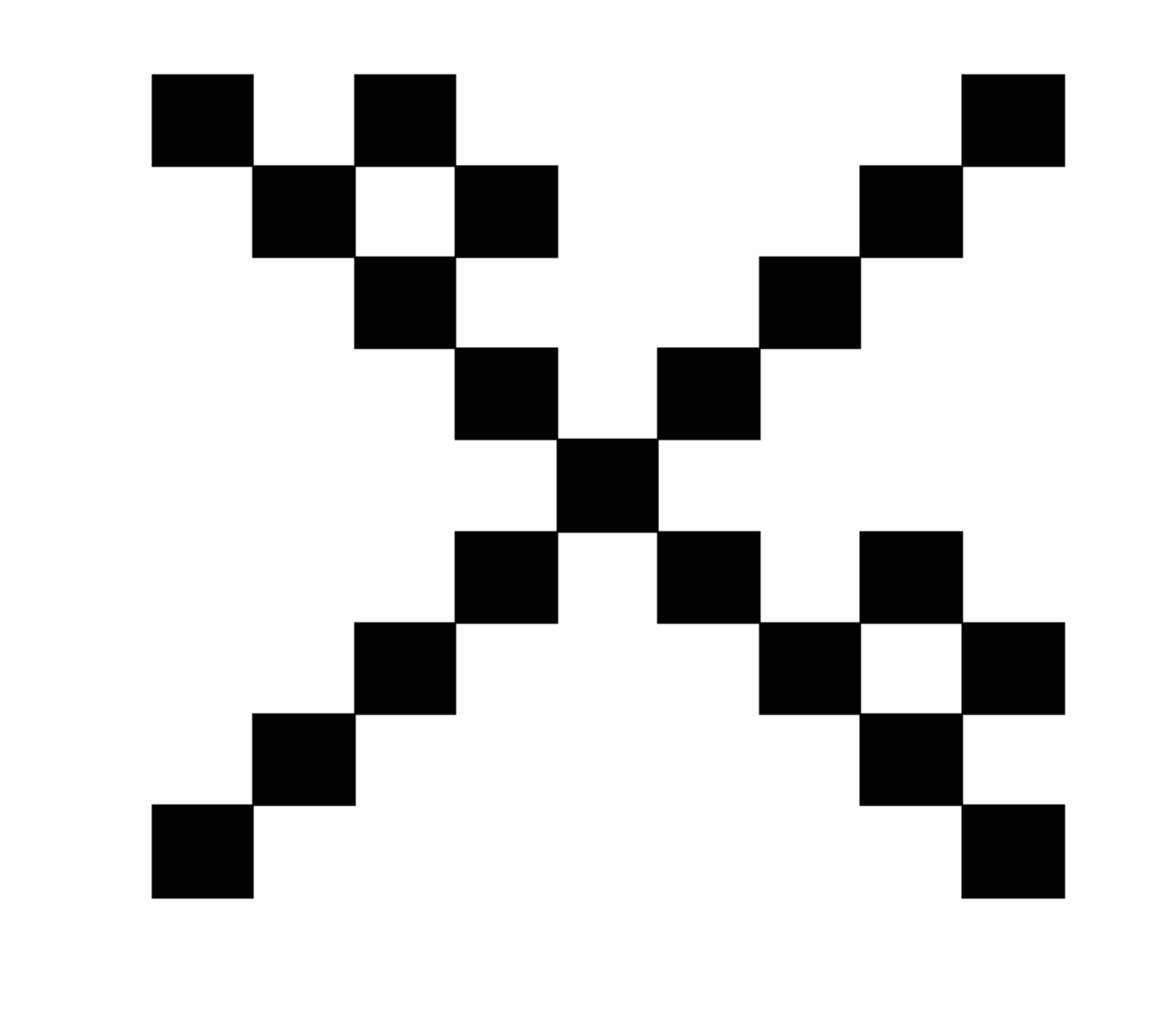}
 \end{minipage}  \qquad \quad
 \begin{minipage}[t]{1.6in}
\includegraphics[width=1.6in]{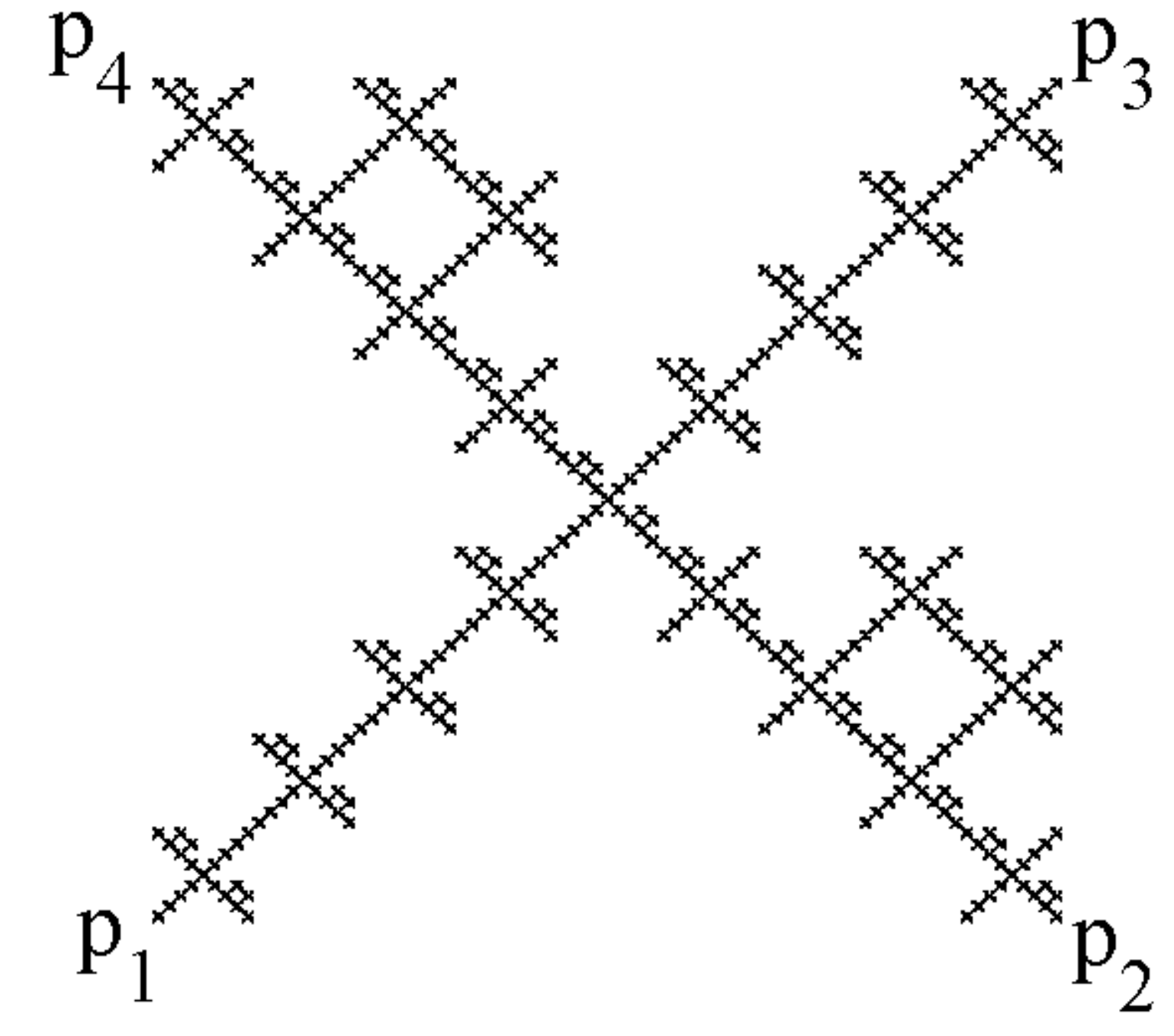}
\end{minipage} &
\end{tabular}
}\caption{The eyebolted Vicsek cross $K$}\label{fig5}
\end{figure}

\vspace {0.1cm}
Let $\{a_i\}_{i=1}^{21}$ be the center of these sub-squares. Let $\{F_i\}_{i=1}^{21}$ be the IFS on $\mathbb{R}^2$ with
$$
F_i(x) = \frac{1}{9}(x-a_i) + a_i,  \qquad 1\leq i \leq 21 \ ,
$$
where $F_i, 1\leq i\leq 9$ correspond to the $9$ sub-squares along $\overline {p_1p_3}$, and the other $F_i$ correspond to the other $12$ sub-squares; let $K$ be the unique nonempty compact set such that
$
K=\bigcup_{i=1}^{21}F_i(K).
$
Then  $\big(K,\{F_i\}_{i=1}^{21}\big)$ is a p.c.f. self-similar set with boundary $V_0=\{p_1,p_2,p_3,p_4\}$. We call this modified Vicsek cross an {\it  eyebolted Vicsek cross}. The Hausdorff dimension of $K$ is $\alpha=\log21/\log9$, and the self-similar measure with the natural weight is the  normalized $\alpha-$dimensional Hausdorff measure $\mu$ on $K$.

\medskip


%
%
In the  \begin{large}${\boxtimes}$\end{large}-X transform in Lemma \ref{th2.7}, we define the vertex set on the X-side by $V_0'= V_0 \cup \{p_0\}$ where $p_0$ is an added point in the center (see Figure \ref{fig2}), and let $V_n'= \bigcup_{|\omega |=n} F_\omega(V_0')$.
\medskip

\begin{lemma}\label{th5.1}
Suppose $(V'_1, \mathbf{\xi})$ is a network with resistance $\mathbf{\xi}=(a,b,c,d)$ on  the four edges of each subcell. Let $\mathbf{\Phi(\xi)} = (a', b', c', d')$ be the trace of $(V'_1, \mathbf{\xi})$ on $V'_0$ , then
\begin{equation}\label{eq5.1}
\mathbf{\Phi(\xi)}=
\left(5a+4c,\ 4b+3d+\varphi ({\mathbf \xi}), \ 4a+5c, \ 3b+4d+\varphi ({\mathbf \xi})\right).
\end{equation}
where $\varphi ({\mathbf \xi}) =  \frac{(b+d)(2a+2c+b+d)}{2(a+b+c+d)}$\ .
\end{lemma}

\medskip

\begin{proof}  The expression  of $\varphi ({\mathbf \xi})$ is obtained by applying the parallel law of the resistances to the two eyebolts on  $V_1'$ (see Figure \ref{fig6}). That $a'= 5a+4c$ is by applying the series law to the branch of $p_1$ to the center, and the same for $b', c', d'$ on the  other three branches.
\end{proof}
\begin{figure}[h]
\centering
\textrm{
\scalebox{0.15}[0.15]{\includegraphics{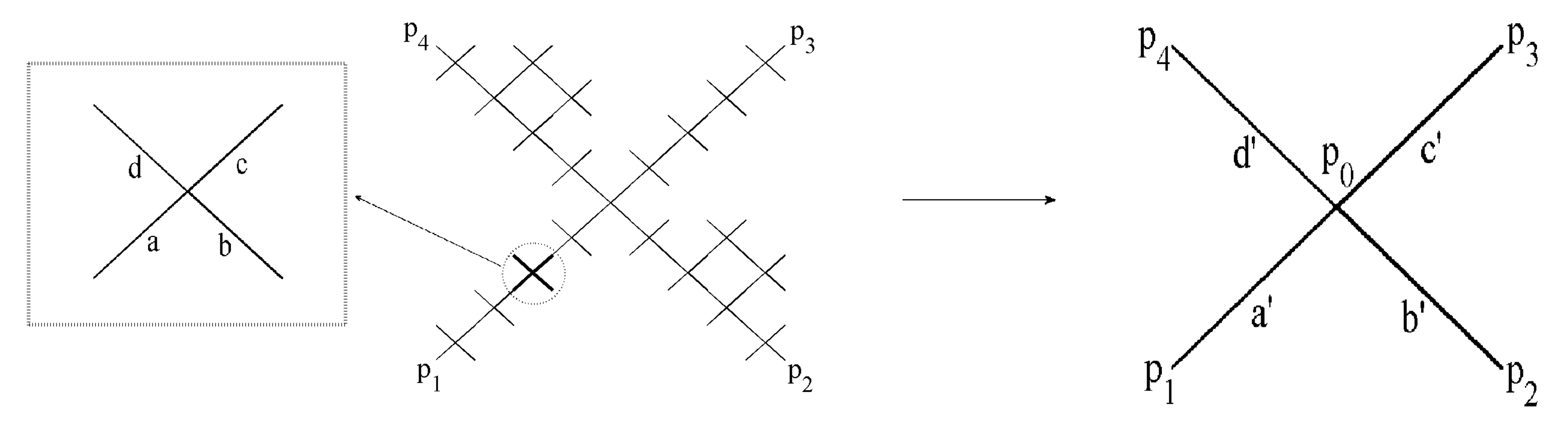}}
}
\caption{Network on $V_1'$ and its trace on $V_0'$}
\label{fig6}
\end{figure}

\medskip

For  $n\geq 0$, we let $G_n$ denote the network on $V_n$ with unit resistance on any two vertices of each subcell $V_\omega$ (i.e., there are $6$ edges  on $ V_\omega$ as in the left picture in Figure \ref{fig2}).

\begin{proposition}\label{th5.2}
Let $R_n(p,q), \ p,q\in V_0$ be the  trace of $G_n$ on $V_0$, then we have

\vspace {0.1cm}
\ (i) $R_n(p_i,p_{i+1})=\frac{a_n+b_n}{2}\asymp a_n$ \  for \ $1\leq i\leq 4, \ p_5 =p_1$;

\vspace {0.1cm}

(ii) $R_n(p_1,p_3)=\frac12(a_n+\frac{a_n^2}{b_n})\asymp\frac{a_n^2}{b_n}$, \ and
\ $R_n(p_2,p_4)=\frac12(b_n+\frac{b_n^2}{a_n})\asymp b_n$,

\vspace {0.1cm}
\noindent where \ $a_n=9^n$\  and \ $b_n=9b_{n-1}-\frac{b_{n-1}^2}{9^{n-1}+b_{n-1}}$
  for $n\geq1$  ($b_0=1$).  Moreover, \ $\lim_{n\to \infty} b_n/9^n =0$, and for any $0<\varepsilon <9$,\ $\lim_{n\to \infty} b_n/ (9-\varepsilon)^n =\infty$.
\end{proposition}

\medskip

\begin{proof} First we use the \begin{large}${\boxtimes}$\end{large}-X transform to convert the resistances from $G_n$ to $G_n'$, it follows that each cell  in $G_n'$ has resistance ${\mathbf \xi} = \frac 14(1,1,1,1)$ (Lemma \ref{th2.7}).  By applying Lemma \ref {th5.1}, we obtain the trace ${\mathbf\Phi(\mathbf \xi)}$ on each cell of $V_{n-1}'$. By induction and some simple calculation,   we see that  the trace of $G_n'$  on $V_0'$ is given by
\begin{equation*}
\mathbf{\Phi}^n({\mathbf \xi})= \frac 14(a_n,b_n,a_n,b_n),
\end{equation*}
where $a_n=9^n$ and $b_n=9b_{n-1}-\frac{b_{n-1}^2}{9^{n-1}+b_{n-1}}$, ($b_0=1$). Then applying the inverse  \begin{large}${\boxtimes}$\end{large}-X transform (Lemma \ref{th2.7}), we obtain the expressions of $R_n(p,q)$ for $p,q\in V_0$ as stated.

To prove the asymptotic values of $b_n$, we let $x_n = \frac {b_n}{9^n}$, then we have
$x_n = x_{n-1} -  \frac {x_{n-1}^2}{9(1+x_{n-1})}$, it follows that $\{x_n\}$  is non-increasing and has a limit $x$ satisfies $x = x-\frac {x^2}{9(1+x)}$. This implies $\lim_{n\to \infty} b_n/9^n =0$.  The limit also implies that for any $0<\varepsilon <9$, $b_n \geq ( 9- \frac \varepsilon 2) b_{n-1}$ for large $n$. and $\lim_{n\to \infty} b_n/ (9-\varepsilon)^n =\infty$ follows. The asymptotic values in (i), (ii) also follows.
\end{proof}

\medskip


  Now we apply the results in Section 3 to conclude the critical exponents of the  eyebolted Vicsek cross $K$, and the density of the Besov space in $C(K)$ and $L^2(K, \mu)$.  First we prove a lemma on the quotient network.

\medskip

\begin {lemma} \label {th5.3} We define a compatible equivalent relation on $V_0$ by identifying $p_2, p_4$, i.e., $V_0=\{p_1\}\bigcup\{p_2,p_4\}\bigcup\{p_3\} := J_1\bigcup J_2\bigcup J_3$. Then

\vspace {0.1cm}

\ (i) the relation satisfies \eqref{eq4.11} in Proposition \ref{th4.4};

\vspace {0.1cm}

(ii)  the trace of $G_n^\sim$ on $V_0^\sim$ is given by
\begin {equation}\label{eq5.2}
R^\sim_n(J_1, J_3) = \frac 12(9^n+ 9^{2n}), \qquad R^\sim_n(J_1, J_2) =R^\sim_n(J_3, J_2) = \frac 14 (1+9^n),
\end{equation}
\end {lemma}

 \medskip

\begin {proof} (i)
 It is easy to see that $J^*_1$ and $J^*_3$ are singletons, $\bar J^*_2$ is the line segment $\overline{p_2p_4}$. Hence $\max\{\ \dim_H(\bar J^*_i):\ i=1,2,3\}= 1 < \alpha$, so that \eqref{eq4.11} holds.

\medskip

(ii) Observe that by identifying $p_2, p_4$, then  $E_0^\sim [u] = E_0[u]$ for  $u \in \ell(V_n^\sim)$. By comparing the conductances, we obtain the resistances on $V_0^\sim$ as  $R^\sim_0(J_1, J_3) =1, R^\sim_0(J_1, J_2) =R^\sim_0(J_3, J_2) = \frac 12$. We make use of the $\Delta$-$Y$ transform in the calculation. Let $x_0, y_0, x_0$ be the corresponding resistances in the $Y$-form. Then
$$
x_0 = \frac 14 \quad \hbox{and} \quad y_0 = \frac 18.
$$
Following the same type of proof in Lemma \ref{th5.1} and Proposition \ref{th5.2} in connection with the quotient (modify Figure \ref{fig6} to a more simple graph), it is easy to show that the trace satisfies
$$
x_n = 9x_{n-1} \quad \hbox {and} \quad y_n=\frac 18.
$$
We then transform  this back to the $\Delta$-form to get $R^\sim_n(J_i, J_j)$ inductively, which yields \eqref{eq5.2}.
\end{proof}

\bigskip

\begin{theorem}\label{th5.4}
For the Besov spaces $B^\sigma_{2, \infty}$ defined on the eyebolted Vicsek cross $K$, the critical exponents are
\begin{equation*}
\sigma^* =\sigma^\# = \frac 12 \Big (1+\frac{\log21}{\log9}\Big ).
\end{equation*}
Moreover,  (i) $B_{2,\infty}^{\sigma^* }$ is dense in $L^2(K,\mu)$; (ii)  $u \in B_{2,\infty}^{\sigma^* }$ takes constant values on each line segment parallel to $\overline {p_2p_4}$, and $B_{2,\infty}^{\sigma^* }$ is not dense in $C(K)$.
\end{theorem}

\medskip

\begin{proof} The cross $K$ has Hausdorff dimension $ \alpha = \frac {\log 21}{\log 9}$;  also
  $R^*  = 9$, and $\lim_{n\to \infty} \frac {R^{*n}}{R_n(p_2, p_4)} = \infty$ (Proposition \ref{th5.2}).  By Theorem \ref{th4.1}, $\sigma^*$ has the expression as asserted, and $B_{2,\infty}^{\sigma^*}$ is not dense in $C(K)$, and $u \in B_{2,\infty}^{\sigma^*}$ has the property as stated.

  \vspace {0.1cm}
  For $\sigma^\#$,  we  have $R^\# = 9$ by Proposition \ref{th5.2}. Take the compatible equivalent relation   in Lemma \ref{th5.3}, then  the conditions in Theorem \ref{th4.3} and Proposition \ref{th4.4} are satisfied. We conclude that $\sigma^\# $ has the same expression as  $\sigma^*$, and  $B^{\sigma^*}_{2, \infty}$ is dense in $L^2(K, \mu)$.
\end{proof}

\bigskip

{\bf 2.  Sierpinski sickle.}
Let $V_0 = \{p_1, p_2, p_3\}$ with $p_1= (0,0), \ p_2 = (1,0), \ p_3 = (\frac 12, \frac {\sqrt 3}2)$. Let $\{F_i\}_{i=1}^{17}$ be the IFS of contractive similitudes on $\mathbb{R}^2$ such that
$$
F_i(x) = \frac 17 x + a_i,  \qquad 1\leq i \leq 17,
$$
where the $a_i$'s are the $17$ points lie on the triangle determined by $V_0$ as indicated in Figure \ref{fig7}.
Let $K$ be the unique nonempty compact set such that
$
K=\bigcup_{i=1}^{17}K_i,
$
and call it a {\it Sierpinski sickle}.  Then $K_i\cap K_j$
contains at most one point, and  satisfies the p.c.f condition. The Hausdorff dimension of $K$ is $\alpha=\log17/\log7$, and the self-similar measure with the natural weight is the  normalized $\alpha-$dimensional Hausdorff measure $\mu$ on $K$.

\begin{figure}[th]
\textrm{
\begin{tabular}{cc}
\begin{minipage}[t]{2.2in} \includegraphics[width=2.2in]{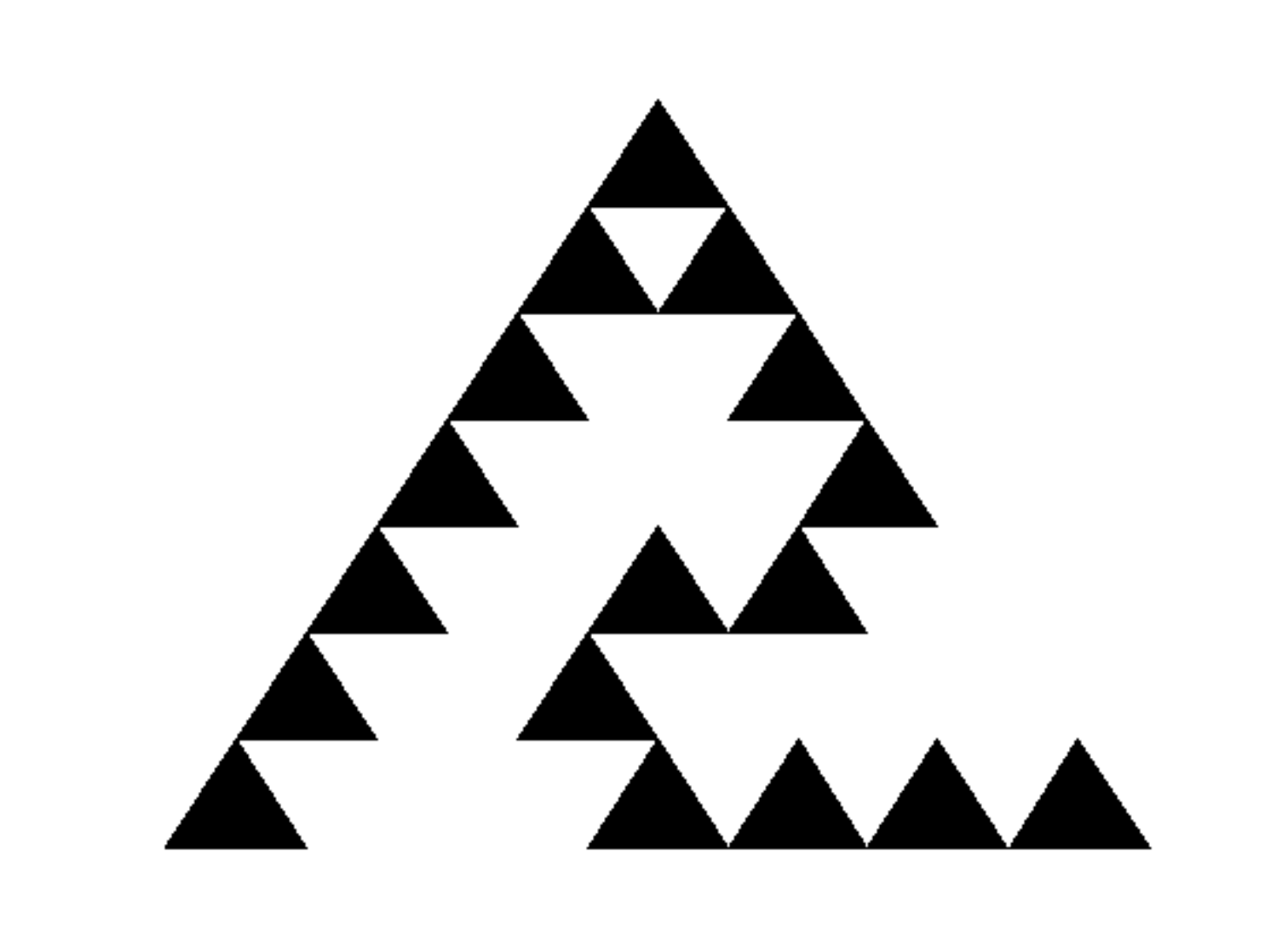}
 \end{minipage} \begin{minipage}[t]{2.2in}
\includegraphics[width=2.2in]{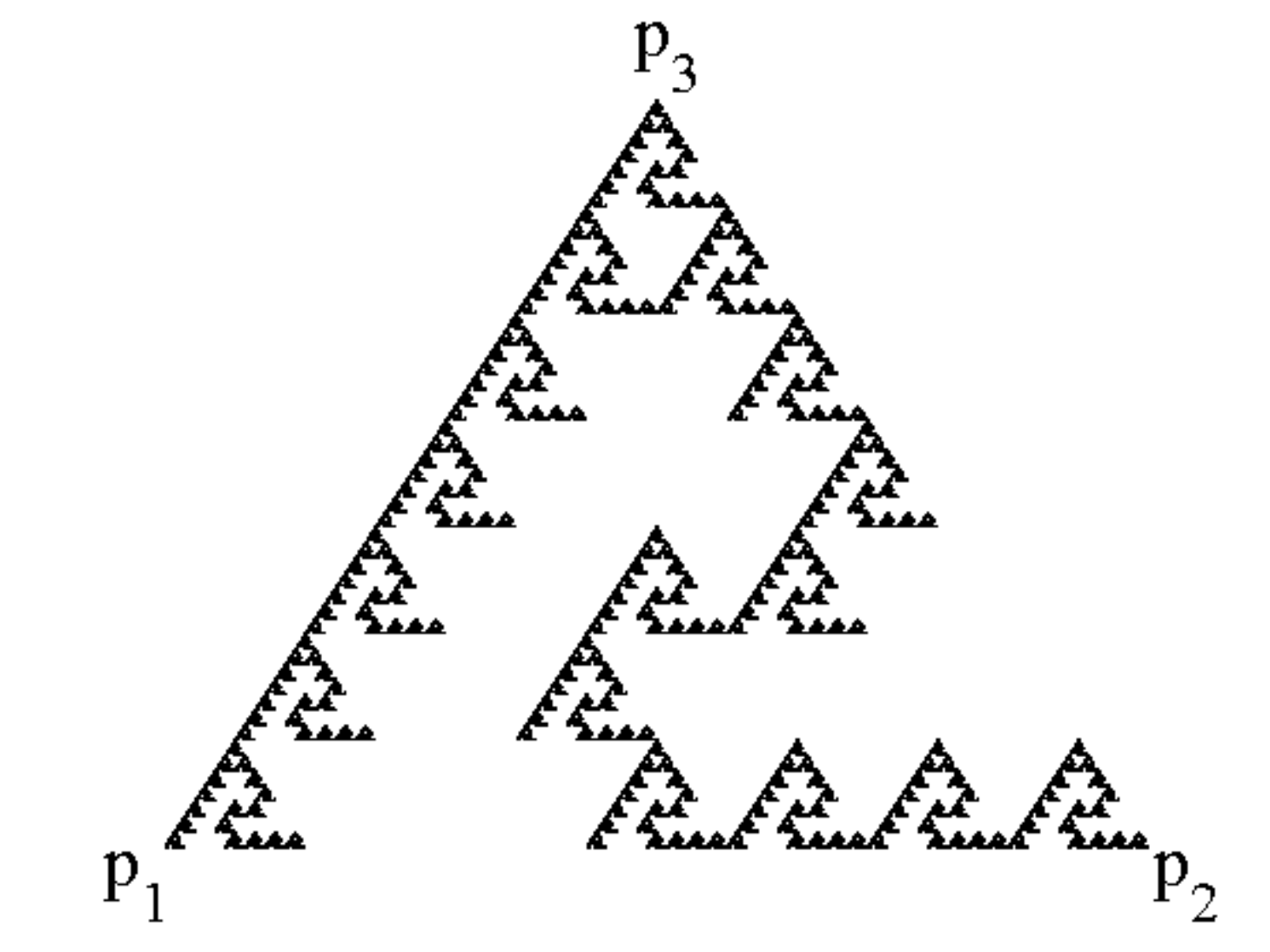}
\end{minipage} &
\end{tabular}
}\caption{The Sierpinski sickle $K$}\label{fig7}
\end{figure}

  On $V_0$, we arrange the three edges clockwise in the order of  $\overline {p_1p_2}$,  $\overline {p_2p_3}$,  $\overline {p_3p_1}$,  and the same way for the sub-triangles in $V_n$. Similar to the last example, we define the vertex sets $V_0'$ an $V_n'$ on the $Y$-side of the $\Delta$-$Y$ transform.
   We show that the  traces $R_n(p,q), p, q \in V_0$ of the primal energy  have different asymptotic geometric rate.
\medskip

\begin {lemma} \label {th5.5}
Suppose $(V'_1, \mathbf{\xi})$ is a network with resistance $\mathbf{\xi}=(a,b,c)$ on  the three edges of each subcell (as indicated in Figure \ref{fig8}). Then the trace of $(V'_1, \mathbf{\xi})$  on  $V'_0$ is:
\begin{equation} \label {eq5.3}
{\mathbf \Phi} ({\mathbf \xi})=(a', b', c')=
 \Big(6a+5c+ \varphi_a,\ 6a+8b+5c+\varphi_b,\ c+\varphi_c\Big).
\end{equation}
where $\varphi_a =\frac{(a+b)(a+c)}{2(a+b+c)}$, and $\varphi_b, \varphi_c$ are defined symmetrically.
\end {lemma}

\medskip

\begin {proof}  We apply the  $\Delta$-$Y$ transform  and obtain \eqref{eq5.3} through a direct calculation (see Figure \ref{fig8}).
\end{proof}
\bigskip
\begin{figure}[h]
\centering
\textrm{
\scalebox{0.15}[0.15]{\includegraphics{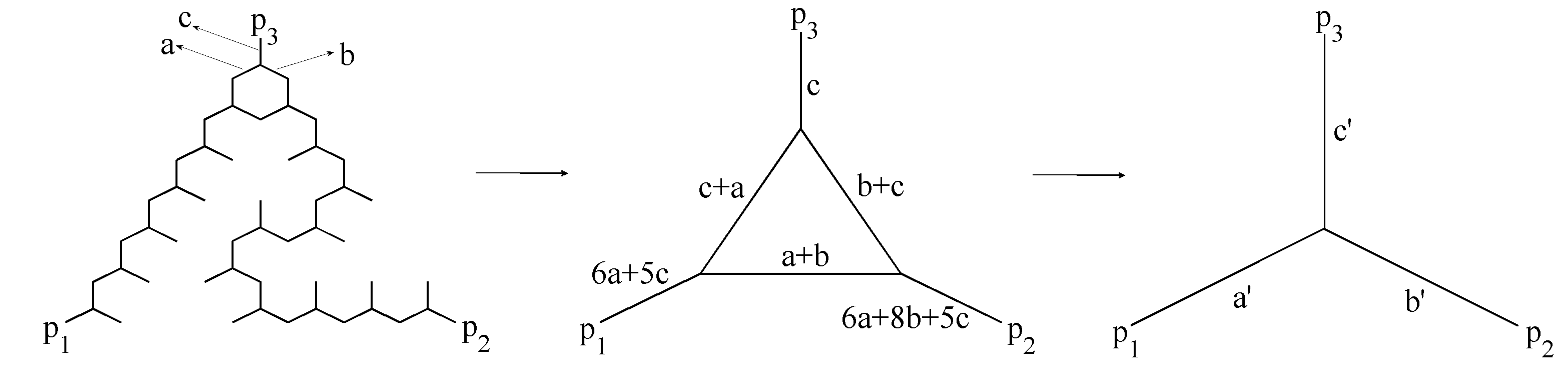}}
}
\caption{The trace of $G_1'$ on $G_0'$}
\label{fig8}
\end{figure}

\medskip

\begin{lemma} \label {th5.6}Let ${\mathbf \Phi}({\mathbf \xi}) = (a', b', c')$ be as in \eqref{eq5.3}. Then

\vspace {0.1cm}
(i)  if  $ c\leq a\leq b$, then $\frac 72  c'\leq a' \leq b'$;

\vspace {0.1cm}

 (ii) if  $\frac 72  c\leq a \leq b$, then
there exists  $ 0<\lambda_0 <1$ such that
$
 \frac ba \leq \lambda_0 \ \frac {b'}{a'} \ .
$
\end{lemma}

\medskip

\begin {proof}
 It is direct to check that $a\leq b$ implies $a'\leq b'$, and $c\leq b$ implies $\frac{a+b}{a+b+c} \geq \frac 12$.  Hence
{\small \begin{equation}\label{eq5.4}
\frac{a'}{c'}
 \geq\frac{6a+5c+\frac{1}{4}(a+c)}{c+\frac{1}{2}(a+c)}\geq \frac72,
\end{equation} }
and (i) follows. By using \eqref{eq5.3} and a simple estimate, we have
{\small \begin{align*}
 {\small \frac{b'}{a'}
\geq\frac{6a+8b+5c+\frac{1}{4}(a+b)}
{6a+5c+\frac{1}{2}(a+c)}
 =\frac{25+ 33\left(\frac{b}{a}+20\left(\frac ca\right)\right)}{22\left(\frac{c}{a} \right) +26}
 \geq \frac{ 33\left(\frac{b}{a}\right)}{22\left(\frac{2}{7} \right) +26}
= \frac {231}{226} \cdot \frac ba\ } \ .
\end{align*}}
The assertion follows by letting $\lambda_0 =\frac{226}{231}$.
\end{proof}

\bigskip
We let $G_n$ denote the electrical network of the primal energy $E_n$ on $V_n$ , and  let $G_n'$ be the corresponding network in the $Y$-form.
As each edge in $G_n$  has resistance $1$, then each edge in $G_n'$ has resistance  $(a_0, b_0, c_0) =\frac 13 (1,1,1)$. Let  $= (a_n, b_n, c_n):= {\mathbf \Phi}^n (a_0, b_0, c_0) $ be the traces of $G'_n$ on $V'_0$.
Also let $R_n(p_1,p_2)$, $R_n(p_2,p_3)$, $R_n(p_3,p_1)$ be the equivalent traces in the $\Delta$-expression.  Then using the $\Delta$-$Y$ transform \eqref{eq2.7}, we have,

\medskip

{\small $R_n(p_1,p_2) =   a_n+b_n+\frac{a_nb_n}{c_n}, \ \
  R_n(p_2,p_3) =   b_n+c_n+\frac{b_nc_n}{a_n}, \ \
  R_n(p_3,p_1) =  c_n+a_n+\frac{c_na_n}{b_n}.
  $}

\bigskip

\begin{proposition}\label{th5.7}
With the above expressions, we have
{\small \begin{align*}
 R_n(p_1,p_2) \asymp  R_n(p_2,p_3)\asymp \left(\frac{17}{2}\right)^n, \quad \hbox {and} \quad
 R_n(p_3,p_1)\asymp 7^n.
\end{align*}}
\end{proposition}

\medskip
\begin{proof}  We will show that  $ b_n \asymp \left(\frac{17}{2}\right)^n$, and $c_n, a_n \asymp 7^n$, and the proposition will follow.
By Lemma \ref{th5.6}, we have
{\small \begin{equation}\label{eq5.5}
\frac{c_n}{b_n}\leq\frac{a_n}{b_n} \leq \lambda_0^n \ (<1).
\end{equation}}
Then making use of this together with
{\small \begin{equation*}
\frac{b_n}{b_{n-1}}=5\cdot\frac{c_{n-1}}{b_{n-1}}+6\cdot\frac{a_{n-1}}{b_{n-1}}+8+ \frac{\left(1+\frac{a_{n-1}}{b_{n-1}}\right)\left(1+\frac{c_{n-1}}{b_{n-1}}\right)
}{2\left(1+\frac{c_{n-1}}{b_{n-1}}+\frac{a_{n-1}}
{b_{n-1}}\right)}
\end{equation*}}
 (by \eqref{eq5.3}),  we obtain $\frac {17}2 - \lambda_0^{n-1} \leq \frac{b_n}{b_{n-1}}\leq\frac{17}{2}
+\frac{23}{2}\lambda_0^{n-1}$,
which implies
 $ b_n \asymp \left(\frac{17}{2}\right)^n$.

 \medskip
 Next,  by using \eqref {eq5.3} and a direct calculation, we have
{\small \begin{align}
\frac{a_{n+1}}{c_{n+1}}-11
   = \frac{\left(2+\frac{13a_n+c_n}{b_n}\right)\Big (\frac {a_n}{c_n}-11\Big ) +
   \frac{132a_n-12c_n}{b_n}}{2\left(1+\frac{a_n+c_n}{b_n}\right)+
   \left(1+\frac{c_n}{b_n}\right)\left(1+\frac{a_n}{c_n}\right)}.\label{eq5.6}
\end{align}}
Letting  $\alpha_n = \Big|\frac {a_n}{c_n}-11\Big|$ and making use of (\ref{eq5.5}),  \eqref {eq5.6} implies
 $$
 \alpha_{n+1} \leq \delta \alpha_n + \gamma \lambda_0^n
 $$
 for some $0<\delta<1$, $\gamma >0$, and for $n$ sufficiently large. An inductive argument shows that  there exist $n_0$ and $0< \delta_1< 1$ such that for $n>n_0$, we have
 $\alpha_{n+1}\leq \delta_1 \alpha_n$. It follows that there is $C_1$ such that
{\small \begin{equation}\label{eq5.7}
\alpha_n\leq C_1\delta_1^n \qquad \forall \ n\geq 0.
\end{equation}}
Now by \eqref{eq5.3} again and simplify the terms,  we have
{\small \begin{align}
\frac{c_{n+1}}{c_n}
=7+\frac{\left(1+\frac{c_n}{b_n}\right)\Big(\frac {a_n}{c_n}-11\Big) -12\frac{a_n}{b_n}}
{2\left(1+\frac{a_n+c_n}{b_n}\right)}, \label{eq5.8}
\end{align}}
By (\ref{eq5.5}), (\ref{eq5.7}) and (\ref{eq5.8}),  we conclude that $c_n \asymp 7^n
$.   The same estimate also holds for $a_n$, and completes the proof.
\end{proof}

\bigskip

\begin{theorem}\label{th5.8} For the Besov spaces $B^\sigma_{2, \infty}$ defined on the Sierpinski sickle $K$, the critical exponents are
\begin{align*}
\sigma^* = \frac 12 \Big (1+\frac{\log17}{\log7}\Big ), \qquad
\sigma^\# =\frac 12 \Big (\frac{2\log17-\log2}{\log7}\Big ).
\end{align*}
Moreover, we have
\vspace{0.1cm}

\ (i) $B^{\sigma^*}_{2, \infty}$ is dense in $C(K)$.

\vspace{0.1cm}

(ii) For $\sigma^* < \sigma \leq \sigma^\#$,  $B^{\sigma}_{2, \infty}$ is dense in $L^2(K, \mu)$.
\end{theorem}

\medskip

\begin{proof} The Sierpinski sickle  $K$ has Hausdorff dimension $ \alpha = \frac {\log 17}{\log 7}$, and $R^*  = 7$ (Proposition \ref{th5.7}).  By Theorem \ref{th4.1}, the expression of  $\sigma^*$ follows. To consider $\sigma^\#$, we know that $R^\# = \frac {17}2$ (Proposition \ref{th5.7}), and we need to check that condition \eqref{eq4.9} in Theorem \ref{th4.3} is satisfied.

\vspace {0.1cm}

For this we take the compatible equivalent relation to be $V_0 = \{p_1,p_3\}\bigcup\{p_2\}:= J_1\bigcup J_2$. Then $R_0(J_1,J_2) = 1/2$, and by a simple inductive argument, we have
\begin{equation} \label{eq5.9}
R^\sim_n(J_1,J_2)= \Big (\frac 12 + 8\Big )\cdot R^\sim_{n-1}(J_1, J_2) = \frac 12 \Big (\frac {17}2\Big )^n.
\end{equation}
Hence  condition \eqref{eq4.9} is satisfied, and $\sigma^\#$ is as asserted.

\vspace{0.1cm}

To prove (i), we use the same  equivalence relation as the above. We check the conditions in Proposition \ref{th4.2}.
First property (B) is satisfied trivially. Next, the closure of the boundary class $J^*_1$ is the line segment $\overline{p_1p_3}$, therefore $\dim_H(\bar J^*_1)=1$ and $\mathcal H^1(\bar J^*_1)<\infty$. We have by Lemma \ref{th3.4}, for a given $u$ on $J_1$, there is extension on $J^*_1$ with $7^{n}E_{J^*_1,n}(u)=E_{J_1,0}(u)$; also $R^{\sim}(J_1, J_2) = \frac {17}2 > R^* =7$.
Thus all the conditions in Proposition \ref{th4.2} are  satisfied and  $B^{\sigma^*}_{2, \infty}$ is dense in $C(K)$.

  \vspace {0.1cm}

   For (ii),  that $B^{\sigma}_{2, \infty}$, $\sigma > \sigma^*$,  is not dense in $C(K)$ is in  the definition of $\sigma^*$. That $B^{\sigma^\#}_{2, \infty}$ is dense in $L^2(K, \mu)$ follows from Proposition \ref{th4.4}, as condition \eqref{eq4.11} holds by the same argument as in Lemma \ref{th5.3},  and the condition \eqref{eq4.12} holds. By the decreasing property of $B^\sigma_{2, \infty}$ on $\sigma$, it follows that of $B^\sigma_{2, \infty}$ is   dense in $L^2(K, \mu)$ for $\sigma < \sigma^{\#}$.
\end{proof}

\medskip
It is well-known that if a p.c.f. set admits a self-similar energy form ${\mathcal E}$, then the re-normalized energy ${\mathcal E }_n[u]$ increases to ${\mathcal E}[u]$, which defines an (equivalent) Besov norm. This is not the case for the
Sierpinski sickle.
\medskip

\begin {corollary} \label {th5.9} On the Sierpinski sickle, there is non-constant $u \in B^{\sigma^*}_{2, \infty}$ such that \\ $\lim\limits_{n\to \infty}\rho^{-(2\sigma^* - \alpha)n} E_n[u] = \lim\limits_{n\to \infty} 7^nE_n[u] =0$.
\end {corollary}
\begin{figure}[h]
\centering
\textrm{
\scalebox{0.16}[0.16]{\includegraphics{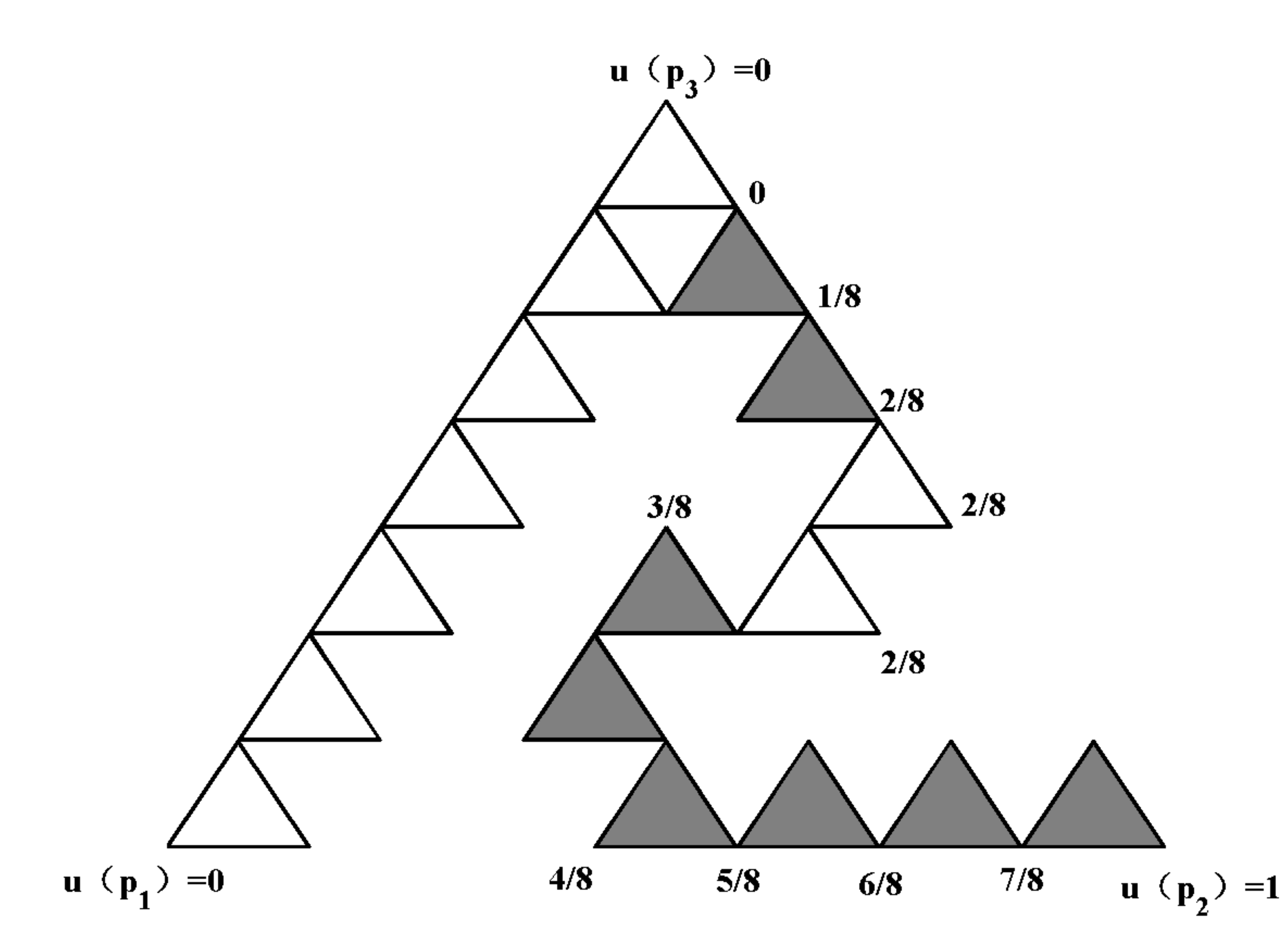}}
}
\caption{$u$ on $V_1$, constant on line segments parallel to $\overline{p_1p_3}$, non-constant on shaded triangles}
\label{fig9}
\end{figure}

\begin {proof}  Let $u$ be defined on $V_0$ with values $u(p_1) =u(p_3) =0$ and $u(p_2) =1$. We extend $u$ on $V_1$ as follows:  take $u$  to be constant $0$ on the $7$ sub-triangles of $F_i(V_0)$ on the left side; the values of the $F_i(V_0)$ in the 10 sub-triangles on the right as in Figure \ref{fig9}.
Note that there are $8$ sub-triangles that $u$ takes non-constant values. Next we define $u$ on $V_2$ by $u_i = u(F_i(p_1))+(u(F_i(p_2))-u(F_i(p_1)))\cdot u\circ F_i^{-1}, 1\leq i \leq 17 $ on $F_i(V_1)$.
 We continue the same process to define $u$ on each $V_n$, and eventually on $V_*$. A direct calculation shows that
$$
7^n {\sum} _{x,y \in F_\omega (V_0): |\omega|=n} |u(x) -u(y|^2 \leq 7^n \cdot 8^n\cdot 2 \cdot \Big (\frac 18\Big )^{2n}.
$$
This implies $\lim_{n\to \infty}  7^n E_n [u]=0$.  Also we have
$$
\sup_{n\geq 0} \Big \{ \rho^{-(2\sigma^* - \alpha)n}  E_n [u] \Big\} \  = \ \sup_{n\geq 0} \big \{ 7^n E_n [u] \big\} < \infty,
$$
by Proposition \ref{th2.4},  $u$ can be extended to be in $B^{\sigma^*}_{2, \infty}$, and proves the statement.
\end{proof}

\bigskip

\section{\bf Constructions of Dirichlet forms}

\bigskip

\subsection{Eyebolted Vicsek cross}\label{sec4.2}
  Since $B^{\sigma^*}_{2, \infty}$ is not dense in $C(K)$,  $B^{\sigma^*}_{2, \infty}$ cannot be the domain of a local regular Dirichlet form on $K$. Nevertheless we will give two constructions of such Dirichlet forms on $K$,  but the domains are different from $B^{\sigma^*}_{2, \infty}$.

\medskip
\begin{theorem} \label {th6.1}  On the eyebolted Vicsek  cross, there are two kinds of (non-primal) local regular Dirichlet forms that can be constructed, one satisfies the energy self-similar identity \eqref{eq2.4},  the other one is from a reverse recursive construction and does not satisfy \eqref{eq2.4}.
\end{theorem}

\begin{proof}
{\it First construction}:  \ We assign two different renormalization factors $\tau', \tau{''}$ (to be determined)  on the cells of $K$ as follows:
let $\tau_1=\tau_2=\cdots=\tau_9=\tau'$ on the $9$ sub-cells $F_1(K),\cdots,F_9(K)$ along the line $\overline{p_1p_3}$, and let $\tau_{10}=\tau_{11}=\cdots=\tau_{21}=\tau{''}$ on the remaining $12$ sub-cells $F_{10}(K), \cdots, F_{21}(K)$;
 then similar to Lemma \ref{th5.1}, we obtain a new trace map $\mathbf{\Phi}_{\tau',\tau``}$ for ${\mathbf \xi} = (a, b, c, d)$:
{\small \begin{align*}
\mathbf{\Phi}_{\tau',\tau{''}}({\mathbf \xi})=
\Big({\tau'}(5a+4c),\ \tau{''}\left(3b+3d+ \varphi ({\mathbf \xi})\right)+{\tau'}b,\
{\tau'}(4a+5c), \ {\tau{''}}\left(3b+3d+\varphi ({\mathbf \xi})\right)+{\tau'}d\Big).
\end{align*}}
where $\varphi ({\mathbf \xi})=\frac{(b+d)(2a+2c+b+d)}{2(a+b+c+d)}$.  Consider the equation
\begin{equation} \label{eq6.1}
\mathbf{\Phi}_{\tau',\tau{''}}(a,b,c,d)=(a,b,c,d),
\end{equation}
i.e., the trace of $G'_1$  coincides with the resistances on $G'_0$. If we apply this to $G'_n$ inductively,  then we obtain a sequence of networks $\{G_k'\}_{k=0}^n$ that is compatible in the sense of Definition \ref{de2.5}, and given the energy self-similar identity.   Specifically, let us take  $a=b=c=d$ in \eqref{eq6.1}, then it reduces to be two simple linear equations, and the solution is
\begin{align*}
  \tau' =\frac19, \quad
  \tau{''} =\frac{16}{135}.
\end{align*}
Let $E_0(u)=\sum_{p,q\in V_0}(u(p)-u(q))^2$,  define
\begin{equation*}
 \mathcal{E}[u]=\lim\limits_{n\to\infty}\sum_{|\omega|=n}{\tau_\omega}^{-1}{E}_0(u\circ F_\omega),
\end{equation*}
and $\mathcal{E}[u]<\infty$ implies that $u\in C(K)$  \cite[Theorem 2.2.6(1)]{K}, thus we can let
$
 \mathcal{F}=\{u\in C(K):\mathcal{E}[u]<\infty\}.
$
Then $(\mathcal{E},\mathcal{F})$  satisfies the self-similar identity
$$
\mathcal{E}[u]= \sum_{i=1}^{17} {\tau_i}^{-1} {\mathcal E}[u\circ F_i], \qquad u \in {\mathcal F}.
$$
It is known that this defines a local regular Dirichlet form on the metric measure space  $(K, d_r,
\nu)$,  where $d_r$ is the resistance metric on $K$, and $\nu$ is the self-similar measure with  weights  $\{\tau_i^s\}_{i=1}^N$ where $\sum_{i=1}^N {\tau_i}^s =1$ (\cite {HMT},  \cite {HW}).

\medskip

{\it Second construction}: \  The main idea is to use  $\mathbf{\Phi}^{-n}$  (where $\mathbf{\Phi}$ is defined in (\ref{eq5.1})) to construct a sequence of  conductances $\{c_n(x,y)\}_n$ in (\ref{eq2.3}) such that ${\mathcal E}_n[u]$ converges for  $u\in C(K)$.

\medskip

Consider the network  $G_n'$, let ${\mathbf y }_n$ be the resistance on each cell of $G_n'$. We are looking for ${\mathbf y }_n= (s_n, t_n, s_n, t_n)$ such that the trace is ${\mathbf y }_0 =(1,1,1,1)$, i.e., $\mathbf{\Phi}^n ({\mathbf y}_n) = {\mathbf y }_0$. As $\mathbf{\Phi}(s,t,s,t)= (9s, 9t- \frac {t^2}{s+t}, 9s, 9t- \frac {t^2}{s+t})$, it follows that
\begin{equation*}
  \mathbf {y}_n=\mathbf{\Phi}^{-n}({\mathbf y}_0)=\left(s_n, t_n, s_n, t_n\right)
\end{equation*}
 where $s_n=9^{-n}$ and $t_{n-1}=9t_n-\frac{t_n^2}{9^{-n}+t_n} (>0)$ for $n\geq1$. Hence by the compatibility of $G_n'$ and $G_0'$ with resistance $\mathbf{y}_n$ and ${\bf y}_0$ respectively, we have
{\small\begin{align}
  &\min_{\upsilon}\Big\{\sum_{|\omega|=n}\Big (s^{-1}_n\sum\limits_{i=1,3}\left|\upsilon\circ F_\omega(p_i)-\upsilon\circ F_\omega(p_0)\right|^2+t^{-1}_n\sum\limits_{j=2,4}|\upsilon\circ F_\omega(p_j)-\upsilon\circ F_\omega(p_0)|^2\Big)\Big\}\notag\\
  &=\sum\limits_{p\in V_0}|u(p)-u(p_0)|^2.\label{eq6.2}
\end{align}}
where the minimum is taken over all $\upsilon \in \ell (V_n')$ such that $\upsilon|_{V_0} = u$.  Then by applying the inverse of\begin {large} ${\boxtimes}$\end{large}-X transform  (Lemma \ref{th2.7}) to each cell in $G'_n$, we obtain an equivalent network   $G_n$.
 \begin{equation} \label {eq6.3}
\min_{\upsilon}\Big\{\sum_{x,y\in V_\omega,|\omega|=n}c_{n}(x,y)\left|\upsilon(x)-\upsilon(y)\right|^2\Big\}
=  \sum_{ p,q\in V_0}\frac14|u(p)-u(q)|^2,
\end{equation}
where the resistances $c_n(x,y)^{-1}$ on $V_n$ are given by
\begin{align*}
  &c_n(F_\omega(p_i),F_\omega(p_{i+1}))^{-1} =2(s_n+t_n),\ i=1,2,3,4,\ (p_5=p_1)\\
  &c_n(F_\omega(p_1),F_\omega(p_3))^{-1} =2(s_n + \frac{s_n^2}{t_n}) , \\
  &c_n(F_\omega(p_2),F_\omega(p_4))^{-1} =2(t_n + \frac{t_n^2}{s_n}).
\end{align*}
 For $u\in C(K)$ and $n\geq0$, let
\begin{equation*}
{\mathcal E}_n[u] = \sum\limits_{x,y\in V_\omega,|\omega| = n }c_n(x,y)|u(x)-u(y)|^2.
\end{equation*}
By the compatibility of ${G_n}$  and $G_{n-1}$ through the ${\bf y}_n$ and ${\bf y}_{n-1}$, we see that ${\mathcal E}_n[u]$ is an increasing sequence on $n$, define
\begin{align*}
{\mathcal E}[u] =\lim\limits_{n\rightarrow\infty} {\mathcal E}_n[u], \quad
{\mathcal F}=\{u\in C(K): {\mathcal E}(u)<\infty\}.
\end{align*}
Note that  $\mathcal F$ is dense in $C(K)$ by approximating $u \in C(K)$ through the piecewise harmonic functions constructed from \eqref {eq6.3} applied to the subcells.  Hence it is not hard to see that $({\mathcal E},{\mathcal F})$ is a local regular Dirichlet form on the metric measure space $(K, |\cdot |, \mu)$.

\medskip

Finally we show by contradiction that the above Dirichlet form  does not satisfy the self-similar identity. Assume that there exist positive numbers $\tau_1,\tau_2,\cdots,\tau_{21}$ such that for any $u\in \mathcal{F}$,
\begin{equation}\label{eq6.4}
\mathcal {E}[u]={\sum}_{i=1}^{21}{\tau_i}^{-1}\mathcal {E}[u\circ F_i].
\end{equation}
Recall that in our construction, the weight we put on each cell is the same, then we  have $\tau_1=\tau_2=\cdots=\tau_{21}=\tau$.

\vspace {0.1cm}
Let $u_1$ be the function that is linear on the line segment $\overline{p_1p_3}$ with boundary values $u_1(p_1)=1$, $u_1(p_3)=0$, and $u_1$ is constant on all the eyebolted branches issued at some point on $\overline{p_1p_3}$.
Then the energy of $u_1$ is supported on $\overline{p_1p_3}$. We can  easily show that $\mathcal {E}_n[u_1]=\frac12$ for all $n\geq0$, and thus $\mathcal {E}[u_1]=\frac12$. Similarly we have $\mathcal {E}[u_1\circ F_i]=\frac12\cdot\frac1{9^2}$ for each $i=1,2,\cdots,9$ along the line $\overline{p_1p_3}$,  and $\mathcal {E}[u\circ F_i]=0$ for the rest twelve maps. By using \eqref{eq6.4},
we obtain $\tau=\frac19$.

\vspace{0.1cm}

Let $u_2$ be the harmonic function with boundary values $\Big(u_2(p_1),u_2(p_2),$ $u_2(p_3),u_2(p_4)\Big)=(0,1,0,0)$. By using (\ref{eq6.4}) $n$ times with $\tau_1=\tau_2=\cdots=\tau_{21}=1/9$, and $u=u_2$, we obtain
\begin{equation}\label{eq6.5}
\mathcal{E}[u_2]=\sum_{|\omega|=n}{\tau_\omega}^{-1}\mathcal {E}[u_2\circ F_\omega]=9^n\sum_{|\omega|=n}\mathcal {E}[u_2\circ F_\omega], \qquad  n>0.
\end{equation}
Since for any $u\in \mathcal F$, $\mathcal {E}[u]$ is the limit of the increasing sequence $\mathcal {E}_n[u]$, the trace estimate yields
\begin{equation}\label{eq6.6}
9^n\sum_{|\omega|=n}\mathcal {E}[u_2\circ F_\omega]
\geq 9^n\sum_{|\omega|=n}\mathcal{E}_0[u_2\circ F_\omega]= 9^n\sum_{p,q\in V_\omega;|\omega|=n}\frac14|u_2(p)-u_2(q)|^2.
\end{equation}
On the other hand, by (\ref{eq4.1}) and Proposition \ref{th5.2}, we have
\begin{align}
9^n\sum_{p,q\in V_\omega,|\omega|=n}\frac14|u_2(p)-u_2(q)|^2
&\geq\frac14\min_{u:u|_{V_0}=u_2}\Big\{9^n\sum_{x,y\in V_\omega,|\omega|=n}\left|u(x)-u(y)\right|^2\Big\}\notag\\
&\geq C^{-1}\frac {a_n}{b_n}|u_2(p_2)-u_2(p_4)|^2\notag\\
&= C^{-1}\frac {a_n}{b_n}\rightarrow\infty \text{ as }n\rightarrow\infty.\label{eq6.7}
\end{align}
Hence we see from (\ref{eq6.5}), (\ref{eq6.6}) and (\ref{eq6.7}) that $\mathcal{E}(u_2)=\infty$, contradicting $u_2\in\mathcal{F}$.
Therefore the Dirichlet form  does not satisfy the energy self-similar identity.
\end{proof}

\bigskip

\subsection {\bf Sierpinski sickle}
Let $K$ be the Sierpinski sickle. Despite $B^{\sigma^*}_{2, \infty}$ is dense in $C(K)$, the primal energy does not give a local regular Dirichlet form in view of Corollary \ref{th5.9}.    We do not know if $B^{\sigma^*}_{2, \infty}$ can be domain of some other local regular Dirichlet form $(\mathcal{E},\mathcal{F})$ on $L^2(K,\mu)$ with $B^{\sigma^*}_{2, \infty}$ as domain.  On the other hand,  we  have the following conclusion.

\medskip
\begin {theorem}  \label {th6.2}The Sierpinski sickle admits a local regular Dirichlet form that satisfies the  energy self-similar identity.
\end{theorem}

\medskip

\begin {proof} We will determine three renormalization factors $\tau_L, \tau_R, \tau_T$ on the cells of $K$ as follows:  let

\vspace{0.1cm}
$\tau_1=\tau_2=\cdots=\tau_5=\tau_{L}$ on the left $5$ sub-triangles $F_1(K),F_2(K),\cdots,F_5(K)$;

\vspace {0.1cm}
$\tau_6=\tau_7=\tau_8=\tau_{T}$ on the $3$ top sub-triangles $F_6(K),F_7(K),F_8(K)$;

\vspace {0.1cm}
$\tau_9=\tau_{10}=\cdots=\tau_{17}=\tau_{R}$ on the right $9$ sub-triangles $F_9(K),F_{10}(K),\cdots,F_{17}(K)$.

\vspace {0.2cm}

\noindent Then similar to Lemma \ref{th5.5}, we obtain the trace map:
{\small \begin{align*}
 \Phi_{\tau_L,\tau_R,\tau_T}(a,b,c) = & \ (a',b',c')\\ \nonumber
 =&\  \Big({\tau_L}(5a+5c)+{\tau_T}{\left(a+\varphi_a\right)}, \ {\tau_R}(6a+7b+5c)+{\tau_T}{\left(b+\varphi_b\right)},\ {\tau_T}{\left(c+\varphi_c\right)}\Big).
\end{align*}}
where $\varphi_a = \frac{(a+b)(a+c)}{2(a+b+c)}$, and define $\varphi_b$, $\varphi_c$ symmetrically.  Let us take  $a=b=kc$ with $k>1$ and solve the equation
\begin{equation} \label{eq6.8}
\Phi_{\tau_L,\tau_R,\tau_T}(a,b,c)=(a,b,c),
\end{equation}
we obtain
\begin{align*}
  \tau_L =\frac{k(k-1)}{5(k^2+6k+3)}, \quad
  \tau_R =\frac{k^2-1}{(13k+5)(k^2+6k+3)}, \quad
  \tau_T =\frac{2(2k+1)}{k^2+6k+3}.
\end{align*}
Let $E_0(u)=(u(p_1)-u(p_2))^2+k(u(p_2)-u(p_3))^2+k(u(p_3)-u(p_1))^2$ on $V_0$,  define
\begin{equation*}
 \mathcal{E}[u]=\lim\limits_{n\to\infty}\sum_{|\omega|=n}{\tau_\omega}^{-1}E_0(u\circ F_\omega),
\end{equation*}
and let
$
 \mathcal{F}=\{u\in C(K):\mathcal{\mathcal E}(u)<\infty\}.
$
Then $(\mathcal{E},\mathcal{F})$ is a regular local Dirichlet form on $L^2(K,\mu)$, and satisfies the self-similar identity
$$
\mathcal{E}[u]= \sum_{i=1}^{17} {\tau_i}^{-1} \mathcal{E}[u\circ F_i], \qquad u \in {\mathcal F}.
$$
\end{proof}

\noindent  {\bf Remark}.  Unlike the eyebolted Vicsek cross, we cannot get the other Dirichlet form on the Sierpinski sickle through the reverse recursive construction. Indeed, for any nonnegative initial value $(a_0, b_0, c_0)$ with $a_0+b_0+c_0=1$ for the sub-triangles in $V_n$ (we can assume this because ${\mathbf \Phi}(\lambda(a_0, b_0, c_0) )=\lambda{\mathbf \Phi}(a_0, b_0, c_0), \ \lambda >0$), let $(a_n,b_n,c_n)={\mathbf \Phi}^n (a_0, b_0, c_0)$ be the trace.
We claim that,
$\frac{b_n}{a_n+c_n}$ goes to infinity very fast, that is
\begin{equation*}
\frac1{a_n+b_n+c_n}(a_n,b_n,c_n)\rightarrow(0,1,0) \ \ \text{ uniformly as } n\rightarrow\infty.
\end{equation*}
Indeed,if $c_0 \leq a_0 \leq b_0$, then by Lemma \ref{th5.6}, we see that there exists $0<\lambda_0<1$ such that for all $n\geq0$,
\begin{equation*}
\frac{b_n}{a_n+c_n}\geq  \frac1{2\lambda_0^n}.
\end{equation*}
If $a_0 \leq c_0 \leq b_0$, then by a direct calculation, $b_1\geq a_1\geq c_1$  and reduces to the previous case.  Finally if $b_0\leq a_0$ (or $b_0\leq c_0$), then
$
\frac{b_1}{a_1}\geq \frac{6a_0+5c_0+8b_0}{6a_0+5c_0+(a_0+b_0)/2}\geq \frac{12}{13},
$
hence
$$
\frac{b_2}{a_2}\geq \frac{6a_1+5c_1+8b_1}{6a_1+5c_1+\frac{a_1+b_1}2}\geq 1.
$$
Also we have $c_2\leq a_2$ by a similar calculation, and hence reduce back to the first case. We checked all the cases and the claim follows.

\medskip

Now if we adopt the same method as in the second construction in Theorem \ref{th6.1}, on the one hand, $\mathbf y_0=(0,1,0)$ is not an interesting choice (as $\mathbf{ y}_{n} = {\mathbf \Phi}^{-n} (\mathbf{ y}_0)=(\frac 2{17})^n\mathbf{ y}_0$ by \eqref{eq5.3}); on the other hand, for any initial value $\mathbf y_0\neq (0,1,0)$, we can not expect to have a non-negative sequence $\{{\mathbf y}_n\}_n$ such that ${\mathbf \Phi} (\mathbf{ y}_n)=\mathbf y_{n-1}$ for all $n>0$.

%

\medskip

\bigskip

\section {\bf Other variances and remarks}

\bigskip

 For the  eyebolted Vicsek cross, if we lift the lower right eyebolt to the upper right position, then the abnormality of the density in Theorem \ref {th5.4} will not appear. We can show that for this new $K$ and with the primal energy, then $\sigma^*=\sigma^\#=\frac{\log21+\log(35/4)}{\log9}$.   As in the first construction in Theorem \ref{th6.1}, we can obtain a self-similar energy form with the renormalization factor  $r = \frac4{35}$ by simply solving equations. By some further computation, we can obtain a local regular Dirichlet form  that  the domain is the associate Besov space $B^{\sigma^*}_{2, \infty}$. We omit the detail.

\bigskip

 For the Sierpinski sickle, if we try to simplify the set by reducing some maps, then we will end up with a homogeneous resistance rate  for $R_n(p, q), \ p, q \in V_0$  for the primal energy, as is shown in the following proposition.

\bigskip

\begin {proposition} \label{th7.1} Consider the self-similar set $K_1$ generated by the IFS  with $15$ maps and contraction ratio $\rho=1/7$ as shown in  Figure \ref{fig10}. Then the relationship of the resistance of the cells on any two levels is given by (as in Lemma \ref{th5.5}):
\small {\begin{equation*}
  {\mathbf \Phi}(a,b,c)=\Big(6a+5c+\varphi_a,\ 5a+6b+4c+\varphi_b,\ c+\varphi_c\Big).
\end{equation*}}
\indent If we let  $(a_n,b_n,c_n)={\mathbf\Phi}^{n}(1,1,1)$  (the $(1,1,1)$ is from the primal energy form on $V_n$), then there exists $\lambda>1$ such that
\begin{equation}\label{eq7.1}
 a_n\asymp b_n\asymp c_n\asymp\lambda^n.
\end{equation}
There exists a Dirichlet form on $L^2(K,\mu)$ satisfying the energy self-similar identity with renormalization factor $\lambda^{-1}$, and the Dirichlet form has $B^{\sigma^*}_{2, \infty}$ as the domain (similar to the Sierpinski gasket).
\end{proposition}

\begin{proof}
It is clear that $c_n\leq a_n\leq b_n$, and using this, it is not hard to show that
$a_n\leq 34c_n$. Furthermore, we claim that $b_n \leq 60a_n$, then $a_n\asymp b_n \asymp c_n$.

\medskip
To prove the claim, by using the fact that $a_n+2b_n+c_n\geq \frac{4}{3}(a_n+b_n+c_n)$, we have
\begin{align}
a_{n+1}+c_{n+1}  =6a_n+6c_n+\frac{(a_n+c_n)(a_n+2b_n+c_n)}{2(a_n+b_n+c_n)}
   \geq \frac{20}{3}(a_n+c_n).\label{eq7.2}
\end{align}
On the other hand,
\begin{align}
  b_{n+1} & = 5a_n+6b_n+4c_n+\frac{(a_n+b_n)(b_n+c_n)}{2(a_n+b_n+c_n)}\notag\\
  & \leq 5a_n+\frac{13}{2}b_n+\frac{9}{2}c_n\leq \frac{13}{2}b_n+5(a_n+c_n).\label{eq7.3}
\end{align}
Combining (\ref{eq7.2}) and (\ref{eq7.3}), and using induction, we obtain
\begin{equation*}
b_n\leq 30(a_n+c_n) \leq 60 a_n,
\end{equation*}
and the claim follows.

\medskip

Note that ${\mathbf \Phi} :{\Bbb R}_+^3\rightarrow {\Bbb R}_+^3$  satisfies  ${\mathbf \Phi}({\bf x})\leq {\mathbf \Phi}({\bf y})$ for any  ${\bf x}\leq {\bf y}$ (coordinate-wise defined), and ${\mathbf \Phi}(c{\bf x})=c{\mathbf \Phi}({\bf x})$ for any $c>0$. For $n,m\geq1$, we have
\begin{align*}
(a_{m+n},b_{m+n},c_{m+n})&=\Phi^{(n+m)}(1,1,1)=\Phi^{(m)}(a_n,b_n,c_n)=b_n\cdot \Phi^{(m)}\left(\frac{a_n}{b_n},1,\frac{c_n}{b_n}\right)\\
&\asymp b_n\cdot \Phi^{(m)}\left(1,1,1\right)=b_n\cdot(a_m,b_m,c_m).
\end{align*}
Then (\ref{eq7.1}) follows by using a sub-additive argument.

\vspace {0.1cm}
 We show that the above $\Phi$ defines a self-similar energy by using a fixed point theorem argument.
 Let  $D=\{(a,b,c):\ a+b+c=1,a, b, c \geq 0\}$ be the simplex, and let ${\widetilde{\mathbf \Phi}}: D\to D$ be the normalization of $\Phi$, i.e.,
 \begin{equation*}
 {\widetilde{\mathbf \Phi}}(a,b,c)=\frac1{11a+6b+10c+\varphi_a+\varphi_b+\varphi_c}\Phi (a, b, c).
 \end{equation*}
For $\varepsilon>0$, let $D_{\varepsilon}=\{(a,b,c)\in D:\ a+c\geq\varepsilon\}$, then we can show by computation that we can choose $\varepsilon$ small enough so that $\widetilde {\mathbf \Phi}$ maps $D_\varepsilon$ to $D_\varepsilon$. By applying the Brouwer's fixed point theorem to $\widetilde{\mathbf \Phi}$ on $D_{\varepsilon}$, we can find  a fixed point $p$ of $\widetilde{\mathbf \Phi}$ on $D_{\varepsilon}$. Obviously $p$ can not be $(1,0,0)$ or $(0,0,1)$.
 Assume that $p=(a_0,b_0,c_0)$ and ${\widetilde{\mathbf \Phi}}(a_0,b_0,c_0)=(a_0,b_0,c_0)$, consequently ${\mathbf \Phi}(a_0,b_0,c_0)=\lambda(a_0,b_0,c_0)$, where $\lambda$ is the same parameter as in (\ref{eq7.1}). Thus by transforming $(a_0,b_0,c_0)$ back to the $\Delta$-form, we get the conductance $(c_0(p_1,p_2),c_0(p_2,p_3),c_0(p_3,p_1))$. We can set
{\small \begin{equation*}
 {\mathcal E}_0(u)=c_0(p_1,p_2)(u(p_1)-u(p_2))^2+c_0(p_2,p_3)(u(p_2)-u(p_3))^2+c_0(p_3,p_1)(u(p_3)-u(p_1))^2.
\end{equation*}}
Set ${\mathcal E}(u)=\lim\limits_{n\rightarrow}\lambda^n\sum\limits_{|\omega|=n}{\mathcal E}_0(u\circ F_\omega)$ and ${\mathcal F}=\{u\in C(K): {\mathcal E}(u)<\infty\}$. $({\mathcal E},{\mathcal F})$ gives the self-similar energy with renormalization factor $\lambda^{-1}$.
Hence $B^{\sigma^*}_{2, \infty}$ is the domain of this Dirichlet form.
\end{proof}

\bigskip

We remark that the  using $D_\varepsilon$ instead of $D$ is to avoid the fixed point $(0,1,0)$. This fixed point $(0,1,0)$ turns out to be repulsive\cite{Me,Sa}.
As in Proposition \ref{th7.1}, the same situation happens if we consider $\rho=1/6$ or $\rho=1/5$ (Figure \ref{fig11} and Figure \ref{fig12}).  For $K_2$,  the resistances  relationship  is (see (\ref{eq5.3}))
 \begin{equation*}
 {\mathbf \Phi} (a,b,c)=\Big(5a+4c+\varphi_a,\ 4a+5b+4c+\varphi_b,\ c+\varphi_c\Big).
\end{equation*}
and for $K_3$,
\begin{equation*}
 {\mathbf \Phi} (a,b,c)=\Big(4a+3c+\varphi_a,\ 2a+4b+3c+\varphi_b,\ c+\varphi_c\Big).
\end{equation*}
We can proceed similar to Proposition \ref{th7.1}.

\begin{figure}[h]
\begin{tabular}{cc}
\begin{minipage}[t]{1.7in} \includegraphics[width=1.7in]{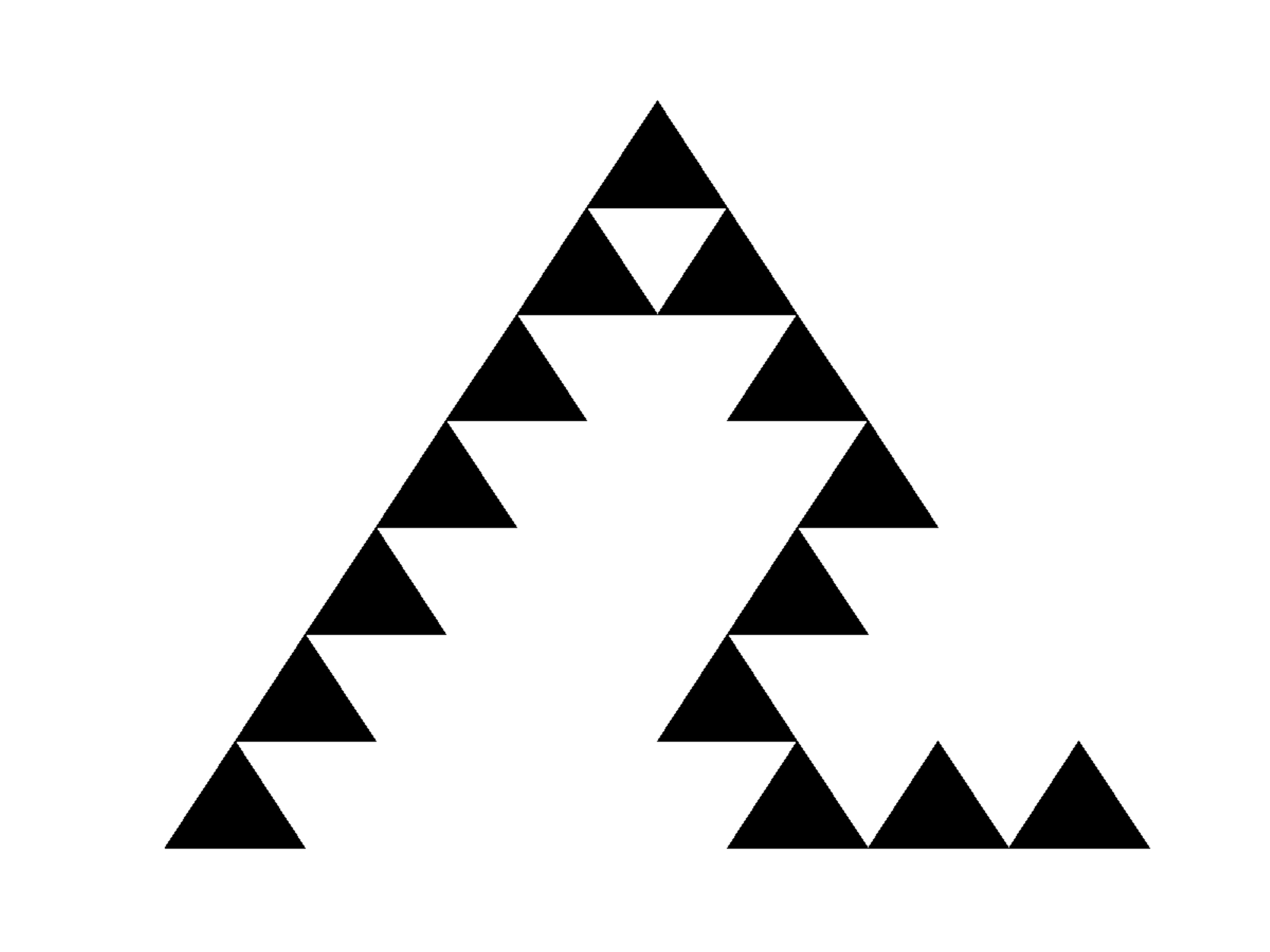}
\caption{$K_1$}\label{fig10}
 \end{minipage} \quad
 \begin{minipage}[t]{1.7in}
\includegraphics[width=1.7in]{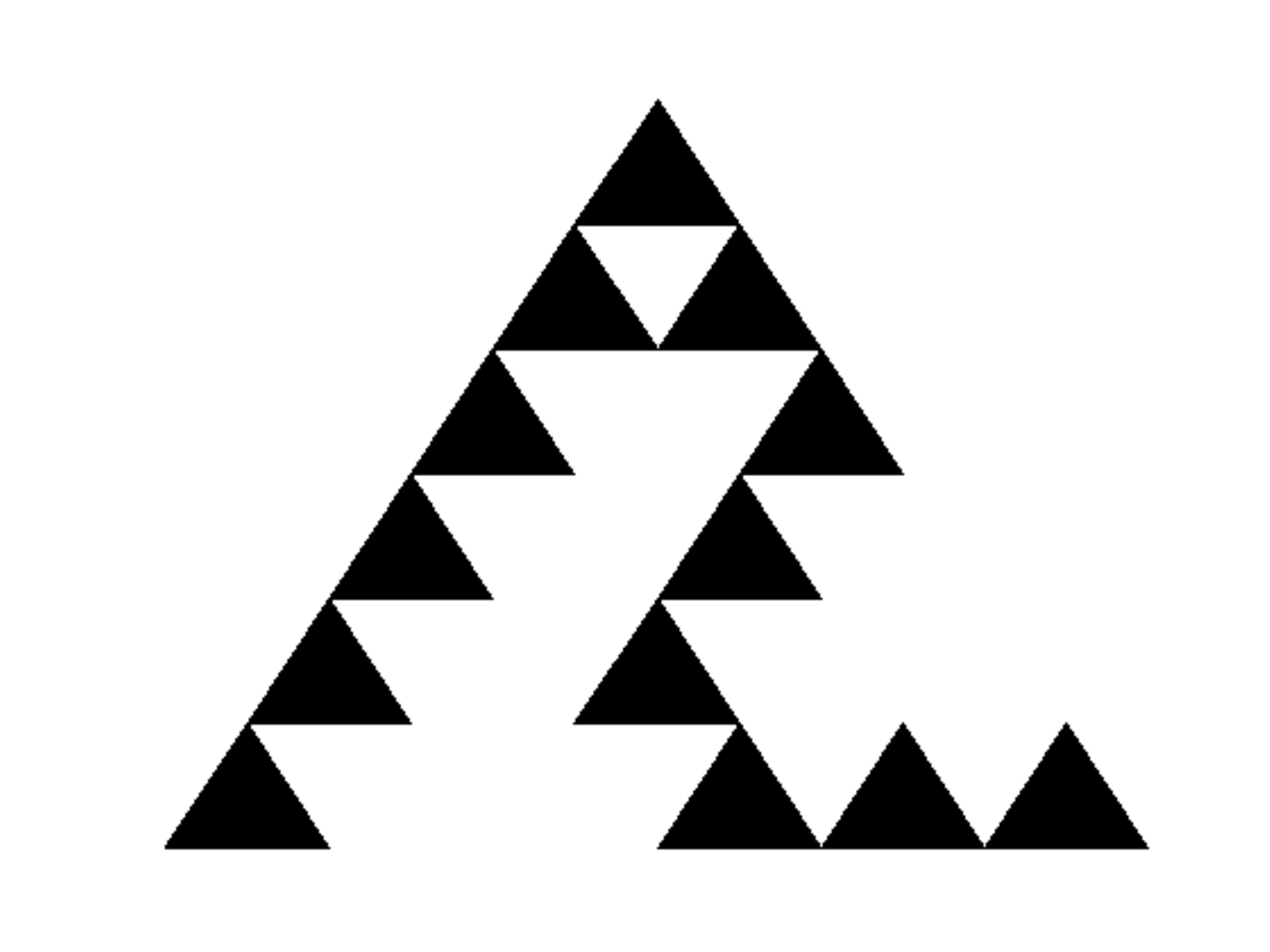}
\caption{$K_2$}\label{fig11}
\end{minipage}
\begin{minipage}[t]{1.7in} \includegraphics[width=1.7in]{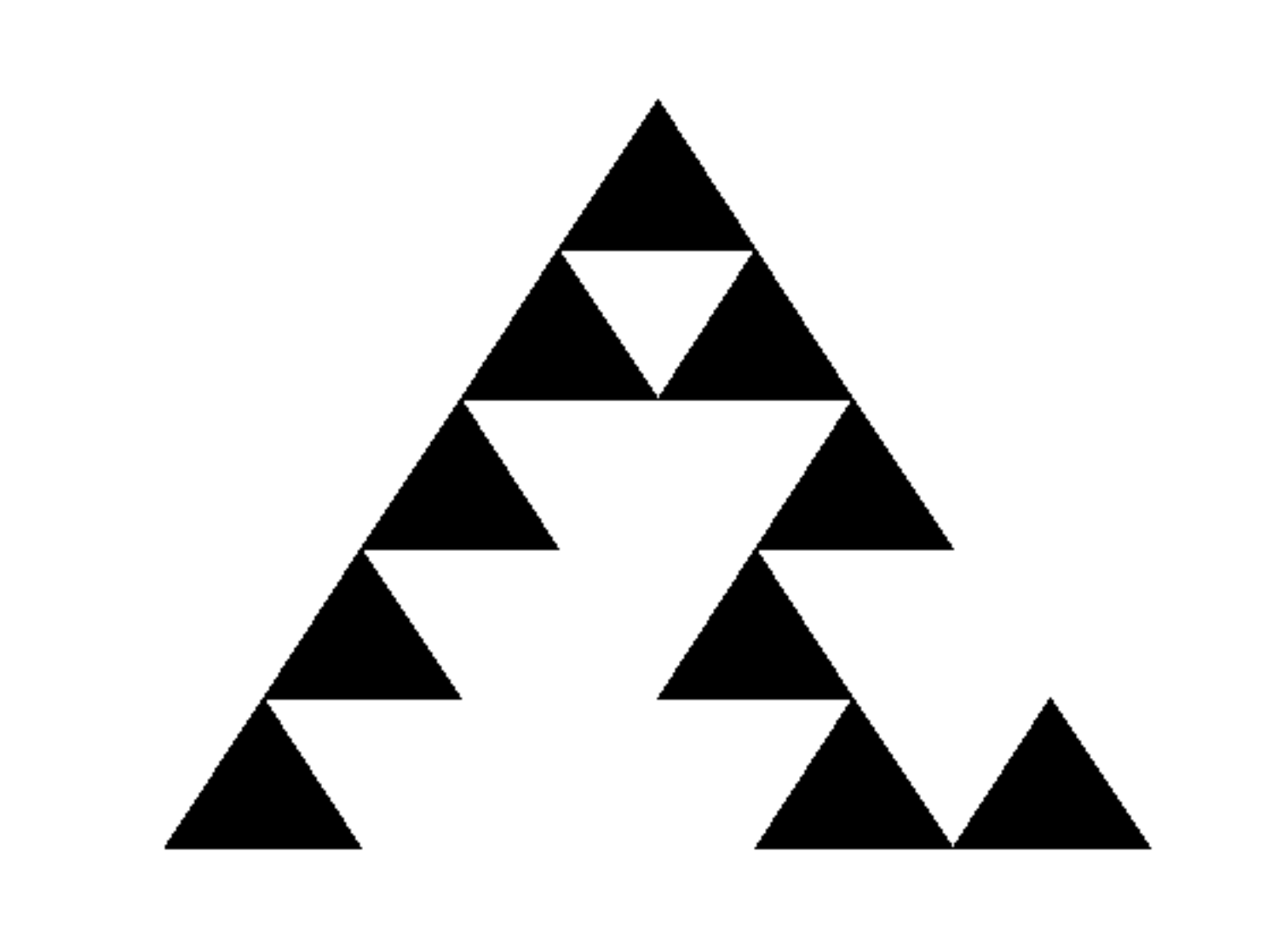}
 \caption{$K_3$}\label{fig12}
 \end{minipage}
 \end{tabular}
\end{figure}

\bigskip

The construction of the energy form by the reverse recursive method was implicitly used by Hattori, Hattori and Watanabe on the Sierpinski gasket \cite{HHW} through a probability consideration, and they call the limit an {\it asymptotically one dimensional diffusion}. This diffusion was investigated further by  Hambly and Kumagai\cite{HK1} on some other nested fractals (see also \cite{HJ,HY}). Recently, in \cite{GLQ}, the authors gave a detail study of this method on the Sierpinski gasket from an analytic point of view; they showed a dichotomy result that for any initial data, the Dirichlet forms obtained are either the standard or the one in \cite{HHW}. They also obtained  sharp estimate of the eigenvalue counting functions of the associated Laplacian with respect to the $\alpha$-Hausdorff measure.  The construction seems to be quite intuitive, but it has limitation (as it fails on  the Sierpinski sickle). It will be interesting to find out the validity of this method  on the more general class of fractals, and to investigate problems related to the associated Laplacian.

\bigskip

Besides the spaces $B^\sigma_{2, \infty}$,  there is another important class of Besov spaces that is associated with the Dirichlet forms. Let
\begin{equation} \label {eq7.4}
[u]^2_{B_{2,2}^\sigma}:=\int_K\int_K\frac{|u(x)-u(y)|^2}
{|x-y|^{\alpha+2\sigma}}d\mu(y)d\mu(x),
\end{equation}
and define $B^\sigma_{2,2}:=\{u\in L^2(K, \mu):||u||_{B^\sigma_{2,2}}<\infty\}$, with norm
$||u||_{B^\sigma_{2,2}}:=||u||_2+ [u]_{B_{2,2}^\sigma}$.  This family of spaces is the domain of some non-local Dirichlet forms, and is associated with the class  fractional Laplacians, and the stable jump processes \cite{CK}. In a recent study of boundary theory of random walks \cite {KLW}, it was shown that  for a class of random walks, the Martin boundary can be identified with the self-similar set $K$ (not necessary p.c.f. set), and the induced Dirichlet form on the boundary has  the expression in \eqref{eq7.4}.  In such setting the critical exponents of the $B_{2,2}^\sigma$ in connection with the random walk has also been studied in \cite{KL}.

\medskip

To put
\eqref{eq7.4} into the previous framework, it is not hard to see that the semi-norm $[u]_{B_{2,2}^\sigma}^2$ is equivalent to
\begin{equation}\label{eq7.5}
\int_0^1\frac{dr}{r}\cdot\frac{1}{r^{\alpha+2\sigma}}\int_K\int_{B(x,r)}(u(x)-u(y))^2d\mu(y)d\mu(x),
\end{equation}
which  can also be expressed as
$
\sum\limits_{n=0}^\infty\rho^{-n(\alpha+2\sigma)}\int_K\int_{B(x,\rho^n)}(u(x)-u(y))^2d\mu(y)d\mu(x).
$
Similar to Proposition \ref{th1.1} (see also \cite {Bo}) , we have the following discretization of $[u]^2_{B^\sigma_{2,2}}$.

\bigskip
\begin {proposition}\label{th7.2}
 Suppose the IFS $\{F_i\}_{i=1}^N$ is as in \eqref{eq2.1} and has the p.c.f. property. Then for $2\sigma> \alpha$,
\begin{equation}\label{eq7.6}
[u]^2_{B^\sigma_{2,2}}\asymp
\sum_{j=0}^\infty \rho^{-(2\sigma-\alpha)j}{\sum}_{x,y\in V_\omega,\ |\omega| = j}\big|u(x)-u(y)\big|^2.
\end{equation}
\end{proposition}

\bigskip

The spaces $B^\sigma_{2,2}$ satisfy the following inclusion relation:  for $0<\varepsilon<\sigma$,
$
B^\sigma_{2,2}\subseteq B^\sigma_{2,\infty}\subseteq B^{\sigma-\varepsilon}_{2,2}.
$
 Hence they share the same critical exponents as the class  $B^\sigma_{2,\infty}, \sigma>0$.   In view of Theorems  \ref{th5.4} and \ref{th5.8}, we have

\medskip

\begin{corollary}\label{th7.3}  For the Sierpinski sickle,
$B_{2,2}^{\sigma^*}$ is not dense in $C(K)$, and $B_{2,2}^{\sigma^\#}$ contains only constant functions. For the eyebolted Vicsek cross, $B^{\sigma^*}_{2, 2}$
contains only constant functions.
\end{corollary}

\medskip
\begin{proof}
 For any $u\in B_{2,2}^{\sigma^*}$, by discretizing $[u]^2_{B_{2,2}^{\sigma^*}}$ as (\ref{eq7.6}), we have
\begin{equation*}
\sum_{m=0}^\infty7^m \sum_{x,y\in V_\omega;|\omega|=m}(u(x)-u(y))^2\asymp[u]^2_{B_{2,2}^{\sigma^*}}<\infty.
\end{equation*}
Then
\begin{equation}\label{eq7.7}
  \lim_{m\rightarrow\infty}7^m \sum_{x,y\in V_\omega;|\omega|=m}(u(x)-u(y))^2=0.
\end{equation}
From this, we claim that $u$ is constant in the direction of $\overline {p_1p_3}$. For if otherwise,  there must be some finite word $\tau$ such that $u\circ F_\tau(p_1)\neq u\circ F_\tau(p_3)$, then for any $n\geq0$,
\begin{equation*}
7^{|\tau|+n} \sum_{|\omega|=|\tau| +n}(u\circ F_\omega(p_1)-u\circ F_\omega(p_3))^2 \geq7^{|\tau|}\left(u\circ F_\tau(p_1)- u\circ F_\tau(p_3)\right)^2>0,
\end{equation*}
a contradiction to (\ref{eq7.7}), and the claim follows. This  implies that $B_{2,2}^{\sigma^*}$ is not dense in $C(K)$.
 We can also show that the rest statements are true by using similar arguments as before, the detail is omitted.
\end{proof}

\bigskip

\section *{\bf Appendix: Graph directed systems}
\bigskip

In here we give a brief supplement to the graph directed systems  of Mauldin and Williams \cite {MW} on the box count and the box dimension to suit our purpose in Theorem \ref{th3.2} and Lemma \ref{th3.4}.

 \bigskip

 Let  $({\mathcal V}, \Gamma)$ be a graph directed system with ${\mathcal V} = \{1, \cdots , N\}$ the set of vertices, and $\Gamma$ the set of edges on ${\mathcal V}\times {\mathcal V}$; let $\Gamma_{i,j}\subset \Gamma$ denote the set of edges from $i$ to $j$.  For each $(i,j)$, let  $\{F_e: e \in \Gamma_{i,j}\}$ be the associated contractive similitudes , then there exists non-empty compact sets $K_i$ such that
 \begin {equation} \label{eq7.8}
 K_i = \bigcup_j \bigcup_{e \in \Gamma_{i,j}} F_e(K_j).
 \end{equation}
 We assume that the family $\{F_e: e\in \Gamma\}$ satisfies the open set condition (OSC), i.e., there exists open sets $U_1, \cdots , U_N$  such that $  \bigcup_j \bigcup_{e \in \Gamma_{i,j}} F_e(U_j) \subset U_i$, and the sets in the union are non-overlapping. We also assume that all the $F_e$ has contraction ratio $\rho$ for simplicity, and that is what we use throughout the paper.

 \medskip

We first consider the graph is strongly connected, i.e.,  for every $i, j\in {\mathcal V}$, there exists a path from $i$ to $j$.  Corresponding to the graph directed system, there is an associated matrix  $N\times N$ matrix  $T = [n_{i, j}]$ where $n(i, j) = \#(\Gamma_{i, j})$.   For ${\bf 1}$ a column vector of $1$'s, $T{\bf 1}$ counts the number of subcells of each $K_i$ (see \eqref{eq7.8}). Let $N_i(n)$ be the number of subcells of the $K_i$ in the $n$-th iteration,  then $N_i(n) =T^n{\bf 1}$ . As $T$ is irreducible, by the Perron-Frobenius Theorem, the maximal eigenvalue $\lambda$ of $T$ is positive, and the eigenvector ${\bf v} > 0$. By using  $ c^{-1} {\bf 1} \leq  {\bf v} \leq c {\bf 1}$ for some $c>0$,  we can show that
$$
 N_i(n)\asymp \lambda^n.
$$
As all the $n$-level cells are of the same size $\rho^n$, and  each one  of them intersects at most $\ell$ of the $n$-level cells for some $\ell >0$ (by OSC), we see that
$$
\alpha =  \lim_{n\to \infty} \frac{\log N_i(n)}{ |\log \rho^n|} = \frac {\log \lambda }{|\log \rho|}
$$
is the box dimension of the $K_i, 1\leq i \leq N$ , which is also the Hausdorff dimension, and $0< {\mathcal H}^\alpha(K_i) <\infty$  \cite {F, MW}.

 \bigskip

If the directed graph is not strongly connected, we assume for simplicity that it has two strongly connected components ${\mathcal V}_1$ and ${\mathcal V}_2$, and we can write $T$ as
$$
 T =  \left[\begin{array}{cc}
                    T_1 &  Q\\
                      0 &  T_2
 \end{array}\right]
$$
where $T_1$ and $T_2$ are irreducible. Let $\lambda_1$, $\lambda_2$ be the eigenvalues of $T_1$ and $T_2$ respectively.  It is clear that   $K_i,\  i \in {\mathcal V}_2$ behaves the same as the above irreducible case.  To consider the $K_i, i \in {\mathcal V}_1$, we observe that
$$
 T^n =  \left[\begin{array}{cc}
                    T^n_1 &  Q_n\\
                      0 &  T^n_2
 \end{array}\right]
$$
with $Q_n = \sum_{k=0}^{n-1}T_1^kQT_2^{(n-1)-k}$. Similar to the above, it is easy to estimate
$
{\bf 1'}^{t} Q_n{\bf 1} \leq C \sum_{k =0}^{n-1}\lambda_1^k\lambda_2^{(n-1)-k},
$
where ${\bf 1'}$ and ${\bf 1}$  are column vector of $1$'s,  and have coordinates equals to $\#({\mathcal V}_1), \#({\mathcal V}_2)$ respectively.

\medskip

 Let $\lambda = \max \{\lambda_1, \lambda_2\}$. If $\lambda_1 \not = \lambda_2$,   then $ N_i(n) \asymp \lambda^n$ for $i\in {\mathcal V}_1$, and  $K_i$ has box dimension $\alpha =  \frac {\log \lambda }{|\log \rho|}$. On the other hand, if  $\lambda_1  = \lambda_2$, then $ N_i(n) \asymp n \lambda^n$ for $i\in {\mathcal V}_1$, and the box dimension is the same. The difference is that in the first case $0 < {\mathcal H}^\alpha (E_i) <\infty$, but in the second case ${\mathcal H}^\alpha (E_i) =\infty$ (which is also the only case that the $\alpha$-Hausdorff measure is infinite) \cite [Theorems 4 and 5]{MW}.

\bigskip
\bigskip
\bigskip

{\it Acknowledgement}. The authors would like to thank Professors D.J. Feng, J.X. Hu and Dr. S.L. Kong for many valuable discussions. They are also indebted to Professor B. Hambly for bringing their attention to a number of references.

\bigskip
\bigskip

\bibliographystyle{siam}

\bigskip
\bigskip

\end{document}